\documentclass[a4paper,10pt]{paper}

\usepackage[utf8]{inputenc}
\usepackage{amsfonts,amsmath,amsthm,amssymb}
\usepackage[all]{xy}
\usepackage{verbatim}
\usepackage{mathrsfs}
\usepackage{yfonts}
\usepackage{a4wide}
\usepackage{color}
\usepackage{wallpaper}
\usepackage{fancyhdr}
\usepackage{subfig}

\title{Random walks on oriented lattices and Martin boundary}
\author{B. de Loynes}

\theoremstyle{plain}
\newtheorem{prp}{Proposition}
\newtheorem{lem}{Lemma}
\newtheorem{thm}{Theorem}

\newtheorem{cor}{Corollary}
\theoremstyle{definition}
\newtheorem{df}{Definition}

\theoremstyle{remark}

\newif\ifdetails
\newcommand{\details}[1]{\ifdetails \textcolor{red}{\textbf{ONLY IN DRAFT VERSION}} \\ \textcolor{blue}{#1} \textcolor{red}{\textbf{NOT EXPECTED IN FINAL VERSION}} \\ \else \fi}

\newif\ifdraft

\newcommand{\fixme}[1]{\ifdraft \marginpar{\textsf{\textbf{FIXME:} \textcolor{red}{#1}}} \else \fi}

\begin{document}

\maketitle

In \cite{Pet1}, transience and recurrence are studied for simple random walks on various types of partially horizontally oriented regular lattices. In this note we aim to give precisions in the transient case by computing the Martin boundary of such random walks.

\section{Notation and definitions}

A directed graph (or di-graph for short) $\mathbf{G}=(\mathbf{V},\mathbf{E})$ is the pair of a countable set $\mathbf{V}$ of vertices and a set $\mathbf{E} \subset \mathbf{V} \times \mathbf{V}$ of directed edges.

Range and source functions, denoted respectively by $r$ and $s$, are defined as mapping $r,s : \mathbf{E} \mapsto \mathbf{V}$, defined for $e=(u,v) \in \mathbf{E}$ by $r(e)=v \in \mathbf{V}$ and $s(e)=u \in \mathbf{V}$. We also define, for each vertex $v \in \mathbf{V}$, its inwards degree by
\begin{displaymath}
d^+_v=\mathsf{card}\{ e \in \mathbf{E} : r(a)=v \},
\end{displaymath}
and its outwards degree by
\begin{displaymath}
d_v^-=\mathsf{card}\{ e \in \mathbf{E} : s(e)=v \}.
\end{displaymath}
The graph $\mathbf{G}$ is said to be transitive if for any vertices $u,v \in \mathbf{V}$ there exists a finite sequence $(w_0, \cdots, w_k)$ of vertices $w_i \in \mathbf{V}$ with $w_0=u$ and $w_k=v$, such that, $(w_i,w_{i+1}) \in \mathbf{E}$ for all $i \in \{ 0, \cdots ,k-1 \}$. We will always suppose the graphs to be  transitive. 

\begin{df}
Let $\mathbf{G}=(\mathbf{V},\mathbf{E})$ be a directed graph. A simple random walk on $\mathbf{G}$ is a $\mathbf{V}$-valued Markov chain $(M_n)_{n \geq 0}$ with Markov kernel $P$ defined by
\begin{displaymath}
P(u,v)=\mathbf{P}(M_{n+1}=v | M_n=u)=\frac{1}{d^-_u}
\end{displaymath}
whenever $d^-_u > 0$, that is $(u,v) \in \mathbf{E}$, and zero otherwise.
\end{df}

In the sequel we will consider two dimensional lattices, i.e $\mathbf{V}=\mathbb{Z}^2$ and $\mathbf{E}$ is a subset of nearest neighborhoods in $\mathbb{Z}^2$. We decompose $\mathbf{V}=\mathbf{V}_1 \times \mathbf{V}_2$ into horizontal and vertical direction. More precisely, if $v \in \mathbf{V}=\mathbb{Z}^2$, then $v=(v_1,v_2)$ with $v_i \in \mathbf{V}_i$ the usual coordinates in $\mathbb{Z}^2$.

Let $\epsilon=(\epsilon_y)_{y \in \mathbf{V}_2}$ be a $\{-1,0,1\}$-valued sequence of variables. The sequence will be defined deterministically, but it can be random variable, or even given by a dynamical system.

\begin{df}
Let $\mathbf{V}=\mathbf{V}_1 \times \mathbf{V}_2$ and $\epsilon$ a sequence as above. We call $\epsilon$-horizontally oriented lattice $\mathbf{G}=(\mathbf{G},\epsilon)$, the directed graph with vertex set $\mathbf{V}=\mathbb{Z}^2$ and edge set $\mathbf{E}$ with the condition $(u,v) \in \mathbf{E}$ if and only if one of the following holds
\begin{enumerate}
\item either $v_1=u_1$ and $v_2=u_2 \pm 1$
\item or $v_2=u_2$ and $v_1=u_1+\epsilon_{u_2}$
\end{enumerate} 
\end{df}

Note that $\mathbf{G}$ is transitive if and only if $1$ and $-1$ are both in the range of $\epsilon$.

Let $\epsilon$ be the sequence defined by $\epsilon_0 = 0$ and $\epsilon_y=\mathsf{sgn}(y)$ where $\mathsf{sgn}$ is the sign function, then, we denote by $\mathbb{H}$ the $\epsilon$-graph induced.

\begin{figure}[!h]
\begin{center}
\includegraphics[width=0.5\textwidth]{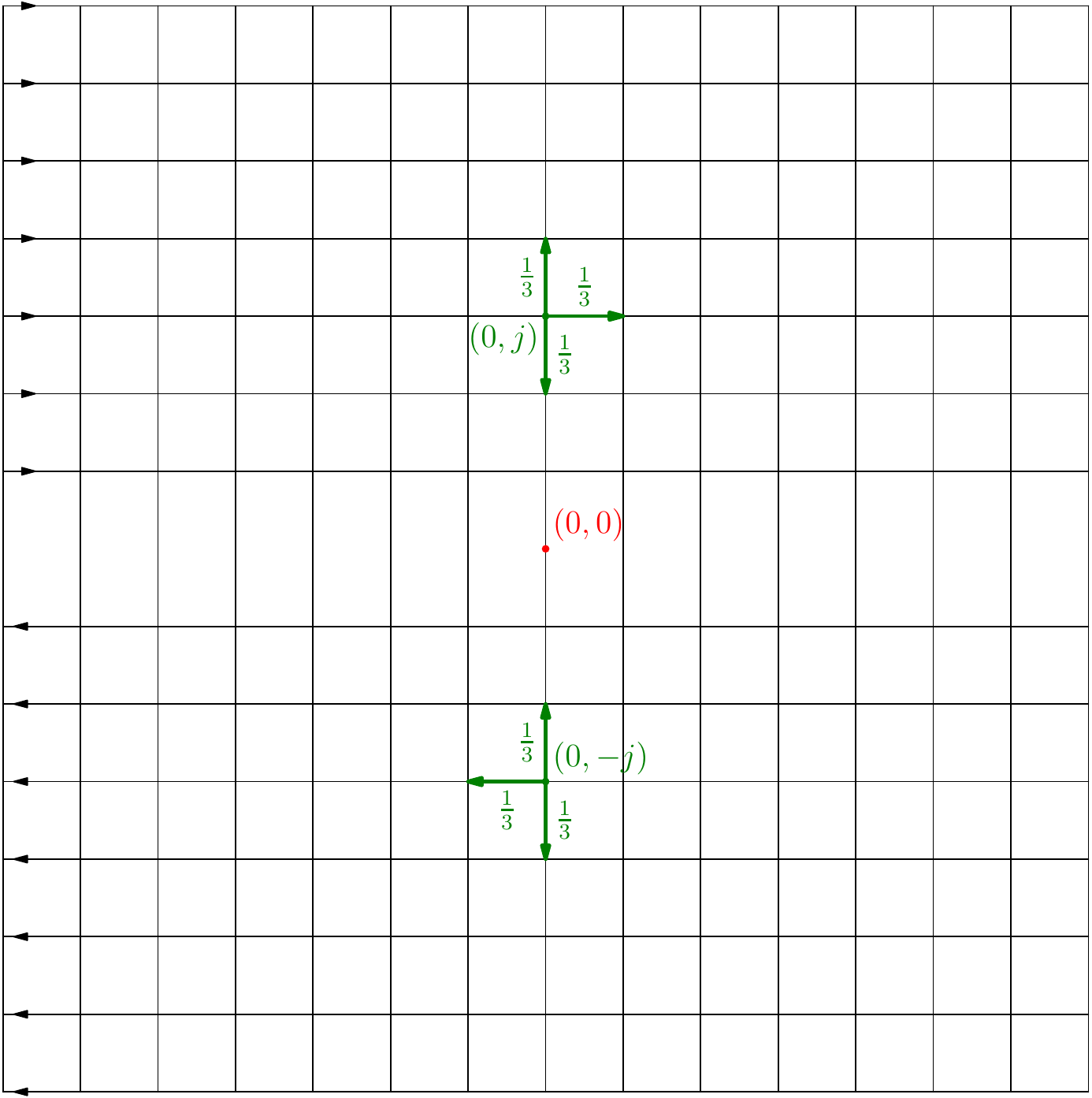}
\end{center}
\caption{The half plane one-way lattice $\mathbb{H}$}
\end{figure}

\section{Results}

Let $(\tau_{n})_{n \geq 0}$ be a sequence of stopping times defined inductively by $\tau_0=0$ and
\begin{displaymath}
\tau_{n+1}=\inf \{ t \geq \tau_n + 1 : M_t^{(2)}=0 \} 
\end{displaymath}
where $M_n=(M_n^{(1)},M_n^{(2)})$, and we have for all $x \in \mathbb{H}$, $\mathbf{P}^x(\tau_n < \infty)=1$.

The sequence of random variables $(M_{\tau_n})_{n \geq 0}$ is itself a Markov chain and will be referred to as the induced Markov chain or the embedded Markov chain. At this step, we may give the main results shown in this paper.

\begin{thm} \label{embedded}
The Martin boundary of the induced Markov chain $(M_{\tau_n})_{n \geq 0}$ is trivial.
\end{thm}

\begin{thm} \label{original}
The Martin boundary of the original Markov chain $(M_n)_{n \geq 0}$ is trivial.
\end{thm}

In the section \ref{secembedded} we will show the theorem \ref{embedded}. The triviality of the Martin boundary comes from precise estimates of the Green kernel computed via the characteristic function of the process $(M_{\tau_n})_{n \geq 0}$. The theorem \ref{original}, proved in section \ref{secoriginal}, is a consequence of similar but more tedious estimates of the Green function. Finally, in a last paragraph, we describe the Poisson boundary of more general random walks and more general partially oriented lattices.

\section{Proofs of theorem}

\subsection{Characteristic function of the induced Markov chain} \label{secembedded}

We start with the computation of the characteristic function of the induced Markov chain $(M_{\tau_n})_{n \geq 0}$.

\begin{df}
Let $(\psi_n)_{n \geq 0}$ be a sequence of independent, identically distributed, $\{-1,1\}$-valued symmetric Bernoulli's variables and
\begin{displaymath}
Y_n = Y_0 + \sum_{k=1}^n \psi_k
\end{displaymath}
for all $n \geq 1$ with $Y_0=M_0^{(2)}$.
Denote by
\begin{displaymath}
\eta_n(y)=\sum_{k=0}^n 1_{\{Y_k=y\}}
\end{displaymath}
\end{df}

\begin{df}
Let $(\sigma_n)_{n \geq 0}$ be a sequence of stopping times defined by induction by $\sigma_0=0$ and
\begin{displaymath}
\sigma_{n+1}=\inf \{ n \geq \sigma_n+1 : Y_n=0 \}, \textrm{ for } n \geq 0.
\end{displaymath}
More precisely, $\sigma_n$ is the $n^{th}$ return time to the origin of a simple symmetric random walk on $\mathbb{Z}$.
\end{df}

\begin{df}
Let $(\xi_n^{(y)})_{n \geq 1, y \in \mathbb{V}_2}$ be a doubly infinite sequence of
independent identically distributed $\mathbb{N}$-valued geometric random variables
of parameters $p$ and $q=1-p$. Let
\begin{displaymath}
X_n=\sum_{y \in \mathbb{V}_2} \epsilon_y \sum_{i=1}^{\eta_{n-1}(y)} \xi_i^{(y)}, n \in \mathbb{N}
\end{displaymath}
Moreover, we denote $|X_n|$ the quantity $\sum_{y \in \mathbb{V}_2} |\epsilon_y| \sum_{i=1}^{\eta_{n-1}(y)} \xi_i^{(y)}, n \in \mathbb{N}$ which represent the total horizontal displacement.
\end{df}

Denote by $T_n$ the time
\begin{displaymath}
T_n=n+\sum_{y \in \mathbb{V}_2} \sum_{i=1}^{\eta_{n-1}(y)} \xi_i^{(y)}
\end{displaymath}
with the convention that the sum $\sum_i$ vanishes whenever $\eta_{n-1}(y)=0$. Then
\begin{displaymath}
M_{T_n}=(X_n,Y_n)
\end{displaymath}

Recall that $\tau_n$ denote the $n^{th}$ return to $0$ of the vertical projection of the $M_n$'s. One has the following.

\begin{prp} \label{law}
The law of $M_{\tau_n}$ is uniquely determinated by the law of $X_{\sigma_1}$, \emph{i.e.} its characteristic function is given by 
\begin{displaymath}
\mathbf{E}^0(e^{i \langle t, M_{\tau_n} \rangle})=\mathbf{E}^0(e^{it_1 X_{\sigma_1}})^n.
\end{displaymath}
We denote by $\phi$ the characteristic function of $X_{\sigma_1}$ with starting point $0$. It is given by
\begin{displaymath}
\phi(t)=\mathbf{E}^0[\exp(i t X_{\sigma_1})]=\mathsf{Re} \textrm{ } r(t)^{-1} g(r(t))
\end{displaymath}
where the functions $g$ and $r$ are defined by the formulae
\begin{displaymath} 
g(x)=\frac{1-\sqrt{1-x^2}}{x} \textrm{ and } r(t)=\frac{p}{1-qe^{it}}.
\end{displaymath}
\end{prp}

\begin{proof}
It is a matter of fact that $\tau_n=\sigma_n + \sum_{i=1}^n |X_{\sigma_n}|=T_{\sigma_n}$. Then,
\begin{displaymath}
\mathbf{E}^0(e^{i \langle t, M_{\tau_1} \rangle}) = \mathbf{E}^0(e^{i t_1 X_{\sigma_1}})
\end{displaymath}
We compute the law of $X_{\sigma_1}$. Denote by $\flat$ the vector $(0,1)$ and factorize by the first step of the random walk, thus
\begin{displaymath}
\begin{split}
\mathbf{E}^0(e^{i t X_{\sigma_1}}) & = \frac{1}{2} \left [ \mathbf{E}^\flat(\exp (i t X_{\sigma_1-1})) + \mathbf{E}^{-\flat}(\exp (i t X_{\sigma_1-1})) \right ] \\
& = \frac{1}{2} \left [ \mathbf{E}^\flat( \exp (i t X_{\sigma_1-1})) + \mathbf{E}^\flat(\exp (-i t X_{\sigma_1-1})) \right ] \\
\end{split}
\end{displaymath}
As a consequence, we only need to compute the following characteristic function
\begin{displaymath}
\begin{split}
\mathbf{E}^\flat[\exp (i t X_{\sigma_1-1})] & = \mathbf{E}^\flat \mathbf{E}^\flat[\exp ( i t X_{\sigma_1-1}) | Y] \\
& = \mathbf{E}^\flat \left [ \prod_{y \in \mathbb{Z}} \prod_{i=1}^{\eta_{\sigma_1-1}} \mathbf{E}^\flat[\exp ( i t \xi_i^{(y)})] \right ] \\
& = \mathbf{E}^\flat[r(t)^{\sigma_1-1}] \\
\end{split}
\end{displaymath}
where $r$ is the characteristic function of the $\xi_i^{(y)}$'s which are i.i.d, geometric random variables, so that $r$ is given by
\begin{displaymath}
r(t)=\frac{p}{1-qe^{it}}
\end{displaymath}


Therefore, we get a closed formula for the characteristic function of $X_{\sigma_1-1}$
\begin{displaymath}
\mathbf{E}^\flat[\exp (i t X_{\sigma_1-1})] = \frac{g(r(t))}{r(t)},
\end{displaymath}
where $g$ is given by $g(x)=\mathbf{E}^\flat[x^{\sigma_1}]$ and satisfies the quadratic relation
\begin{displaymath}
g(x)=\frac{x}{2} ( 1+g(x)^2),
\end{displaymath}
so that $g(x)=\frac{1-\sqrt{1-x^2}}{x}$.
\end{proof}

\subsection{Martin boundary of the induced random walk}
\label{green-est}

By inverse Fourier transform, we find a close formula for the Green function of the induced random walk, namely
\begin{displaymath}
G(x,y)=\pi^{-1} \int_0^\pi \frac{\cos((y-x)t)}{1-\phi(t)} dt
\end{displaymath}
and we want to get an equivalent as $y \to \infty$. It appears that the function $[1-\phi]^{-1}$ has an integrable singularity for $t=0$. The fruitful idea is to separate this singularity from the regular part of the function.

\begin{prp} \label{decinitial}
There exists two analytic functions $a,b$ in a neighborhood of $0$ such that
\begin{displaymath}
\frac{1}{1-\phi(t)}=\frac{c}{\sqrt{|t|}} + \sqrt{|t|} a(t) + b(t)
\end{displaymath}
\end{prp}

The proof of this proposition is postponed to section \ref{technical}. Having this decomposition in mind, a simple computation yields a fine estimate of the integral involved in the formula of the Green function.

\begin{prp} \label{equivalent}
Denote by $\gamma$ the function defined by
\begin{displaymath}
\gamma(x)=\int_0^\pi \frac{\cos(xt)}{1-\phi(t)} dt.
\end{displaymath}
Then, the limit of $\sqrt{x} \gamma(x)$ as $x \to \infty$ exists and is non zero.
\end{prp}

\begin{proof}[Proof]
Denote by $R_a$ and $R_b$ the convergence radii of $a$ and $b$ and choose $\epsilon > 0$ such that $\epsilon < R_a \wedge R_b$, then
\begin{displaymath}
\gamma(x)=\int_0^\pi \frac{\cos(xt)}{1-\phi(t)} dt = \int_{0}^\epsilon \frac{\cos(xt)}{1-\phi(t)} dt + \int_{\epsilon}^\pi \frac{\cos(xt)}{1-\phi(t)} dt
\end{displaymath}
The second terms behaves like $\mathscr{O}\left (\frac{1}{x} \right )$ at infinity because on $(\epsilon,\pi)$ the function $\frac{1}{1-\phi}$ is infinitely continuously differentiable.

Because of the proposition \ref{decinitial}, the first integral term can be split in three parts $\gamma_0, \gamma_1, \gamma_2$. Then,
\begin{displaymath}
\gamma_0(x)= c \int_0^\epsilon \frac{\cos(xt)}{\sqrt{t}} dt,
\end{displaymath}
and setting $u=xt$ we get
\begin{displaymath}
\gamma_0(x)=\frac{c}{x} \int_0^{\epsilon x} \sqrt{x} \frac{\cos(u)}{\sqrt{u}} du.
\end{displaymath}
The latter is a convergent integral so that, when $x \to \infty$, $\gamma_0(x) \sim \frac{c'}{\sqrt{x}}$ with
\begin{displaymath}
c'=c \int_0^\infty \frac{\cos(u)}{\sqrt{u}} du
\end{displaymath}

Secondly, $\gamma_2(x)$ behaves like $\mathscr{O}\left ( \frac{1}{x} \right )$ at infinity. Indeed,
\begin{displaymath}
\gamma_2(x)=\int_0^\epsilon \cos(xt) b(t) dt
\end{displaymath}
and $b$ is infinitely continuously differentiable.

Finally, it remains to estimate the last term which is
\begin{displaymath}
\gamma_1(x)=\int_0^\epsilon \cos(xt) \sqrt{t} a(t) dt
\end{displaymath}
we may integrate by part,
\begin{displaymath}
\gamma_1(x) =\left [ \sqrt{t} a(t) \frac{\sin(tx)}{x} \right ]_0^\epsilon - \frac{1}{x} \int_0^\epsilon \left [ \frac{a(t)}{2\sqrt{t}} + \sqrt{t}a'(t) \right ] \sin(tx) dt
\end{displaymath}
and it follows that $\gamma_1$ behaves like $\mathscr{O} \left (\frac{1}{x} \right )$ and the proposition is proved.
\end{proof}

Finally, we give the proof of theorem \ref{embedded}.

\begin{proof}[Proof of theorem \ref{embedded}]
If we denote by $G_0$ the Green kernel of the Markov chain $(M_{\tau_n})_{n \geq 0}$ then we get for all $x,y \in \mathbb{Z} \times \{ 0 \}$
\begin{displaymath}
G_0(x,y)=\gamma(y-x)
\end{displaymath}
so that the Martin kernel is given by
\begin{displaymath}
K_0(x,y)=\frac{G_0(x,y)}{G_0(0,y)}=\frac{\gamma(y-x)}{\gamma(y)}
\end{displaymath}

By proposition \ref{equivalent}, we have $\gamma(y) \sim \frac{c}{\sqrt{|y|}}$, consequently, for all unbounded sequences $(y_k)_{k \geq 0}$ of points of $\mathbb{Z}$, the limit of $K(x,y_k)$ is equal to 1 as $k$ goes to infinity. Therefore, the Martin compactification is the one point compactification.
\end{proof}

\section{Martin boundary of the original Markov chain} \label{secoriginal}

In this section, we will prove the triviality of the Martin boundary of the original Markov chain $(M_n)_{n \geq 0}$.

Denote by $\nu_x$ the probability, supported by $\mathbb{H}_0=\mathbb{Z} \times \{ 0 \}$, defined by
\begin{displaymath}
\nu_x(z)=\mathbf{P}^x(M_{\tau_1}=z).
\end{displaymath}
Then, strong Markov property implies the following,
\begin{equation} \label{decomposition}
K(x,y)=\frac{\mathbf{E}^x(\eta_{0,\tau_1}(y))}{G(0,y)} + \sum_{z \in X_0} \nu_x(z) K(z,y)
\end{equation}
for $x,y \in \mathbb{H}$.

In section \ref{mk-conditioned}, we show --- corollary \ref{noyaufort} --- that the second term in equation \ref{decomposition} goes to $1$ as $|y|$ goes to infinity for all $x \in \mathbb{H}$, whereas in section \ref{before} the first term will be shown to vanish as $|y|$ goes to infinity.

\subsection{Martin kernel conditioned by the first return time to $\mathbb{H}_0$} \label{mk-conditioned}

We first express the Martin kernel $K(z,y)$ in terms of Fourier transform for $z \in \mathbb{H}_0$.

\begin{prp} \label{expnoyaufaible}
Let $z \in \mathbb{H}_0$ and $y \in \mathbb{H}$, then the Martin kernel is given by
\begin{displaymath}
K(z,y)=\frac{\int_{-\pi}^\pi e^{ity_1-itz} \frac{g(r(t))^{|y_2|}}{1-\phi(t)} dt}{\int_{-\pi}^\pi e^{ity_1} \frac{g(r(t))^{|y_2|}}{1-\phi(t)} dt}
\end{displaymath}
where $g$ is given by
\begin{displaymath}
g(x)=\frac{1-\sqrt{1-x^2}}{x}
\end{displaymath}
and $r$ is given by
\begin{displaymath}
r(t)=\frac{1}{3-2e^{it}}.
\end{displaymath}
\end{prp}

\begin{proof}
If $y=(y_1,y_2) \in \mathbb{H}$ then we will denote by $\bar{y}$ the vector $\bar{y}=(y_1,-y_2)$. Using the geometry of the lattice $\mathbb{H}$, it is easy to see that
\begin{displaymath}
G(z,y)=G(\bar{y},z)=\sum_{w \in \mathbb{H}_0} \nu_{\bar{y}}(w) G_0(w,z)
\end{displaymath}
and
\begin{displaymath}
G(0,z)=G(\bar{y},0)=\sum_{w \in \mathbb{H}_0} \nu_{\bar{y}}(w) G_0(w,0),
\end{displaymath}
for $z \in \mathbb{H}_0$ and $y \in \mathbb{H}$.

Consequently, using the translation invariance of $G_0$ and applying the substitution $v=w-z$ in the first sum, we get
\begin{displaymath}
K(z,y)=\frac{\sum_{v \in \mathbb{H}_0} \nu_{\bar{y}-z}(v) G_0(v,0)}{\sum_{v \in \mathbb{H}_0} \nu_{\bar{y}}(v) G_0(v,0)}.
\end{displaymath}

Recall that $\nu_y(v)=\mathbf{P}^y(M_{\tau_1}=v)=\mathbf{P}^{(0,y_2)}(M_{\tau_1}=v-y_1)$, thus we can assume that $y=(0,y_2)$ and compute,
\begin{displaymath}
\nu_y(v)=\frac{1}{2\pi} \int_{-\pi}^\pi e^{-itv+ity_1}\phi^{y_2}(t) dt
\end{displaymath}
where $\phi^{y_2}$ is given by
\begin{displaymath}
\phi^{y_2}(t)=\mathbf{E}^{y_2}(e^{itX_{\sigma_1}})=g(r(t))^{|y_2|},
\end{displaymath}
and this comes from a simple modification of the computations of the proof of the proposition \ref{law}. Then, let us compute the sum
\begin{displaymath}
\sum_v \nu_{\bar{y}-z} G_0(0,v) = \frac{1}{2\pi} \int_{-\pi}^\pi \phi^{-y_2}(t) e^{ity_1-itz} \sum_v e^{itv}G_0(0,v) dt
\end{displaymath}
and the summation is the Fourier series of the function $[1-\phi(t)]^{-1}$ computed in the section \ref{secembedded}.

As a consequence, we have to estimate the rate of convergence of the integral
\begin{equation} \label{intgreen}
\int_{-\pi}^{\pi} e^{ity_1-itz} \frac{\phi^{y_2}(t)}{1-\phi(t)} dt 
\end{equation}
when $y=(y_1,y_2)$ goes to infinity, that is when $|y_1|$ or $|y_2|$ goes to infinity.
\end{proof}

In the spirit of section \ref{green-est}, we first compute --- see section \ref{technical} --- an analytic decomposition of the characteristic function of the Green function (centered on $\mathbb{H}_0$).

\begin{prp} \label{dec}
The function $g \circ r$ can be decomposed in a neighborhood of 0 as follows
\begin{displaymath}
g(r(t))=1-2\sqrt{|t|}e^{\textsf{sgn}(t)i \frac{\pi}{4}} - \sqrt{|t|} \alpha(t) - \beta(t),
\end{displaymath}
where $\alpha$ and $\beta$ are analytic functions in a neighborhood of 0, satisfying $\alpha(0)=\beta(0)=0$.
\end{prp}

We will estimate the rate of convergence of the integral (\ref{intgreen}). This rate depends on relative rate of escape to infinity of $y_1$ with respect to $y_2$. It is straightforward to show that there are two cases depending on the ratio $\frac{y_1}{y_2^2}$ :
\begin{itemize}
\item $\lim \frac{y_1}{y_2^2} = \lambda \in \mathbb{R}$
\item $\lim \frac{y_1}{y_2^2} = \pm \infty$
\end{itemize}

The first case will be proved in proposition \ref{lemy2} whereas the last one will be handled in proposition \ref{lemy1}.

\begin{prp} \label{lemy2}
Assume that $(y_1,y_2)$ goes to infinity in such a way that $\lim y_1 y_2^{-2} = \lambda \in \mathbb{R}$. Then the sequence
\begin{displaymath}
\left (|y_2| \int_{-\pi}^\pi e^{ity_1-itz} \frac{\phi^{y_2}(t)}{1-\phi(t)} dt \right )_{(y_1,y_2) \in \mathbb{Z}^2}
\end{displaymath}
converges to a non zero constant.
\end{prp}


\begin{proof}
Let $n$ be a positive integer and set $m=y_1-z$, we begin to estimate the difference
\begin{displaymath}
D(t)=\frac{\phi^n(tn^{-2})}{n(1-\phi(tn^{-2}))} - Q(t)
\end{displaymath}
where $Q$ is given by
\begin{displaymath}
Q(t)=\frac{c\exp \{ -2e^{\textsf{sgn}(t) i \frac{\pi}{4}}\sqrt{|t|} \} }{\sqrt{|t|}}
\end{displaymath}
where $\textsf{sgn}$ is the function sign and $c$ is the constant involved in the proposition \ref{decinitial}.

Let $\epsilon > 0$ be sufficiently small so that the decompositions in propositions \ref{dec} and in \ref{decinitial} are satisfied. Then for $|tn^{-2}|<\epsilon$ we have
\begin{displaymath}
\begin{split}
\frac{\phi^n(tn^{-2})}{n(1-\phi(tn^{-2}))} - Q(t) & = \exp \left \{ n \log(1-2e^{\textsf{sgn}(t) i \frac{\pi}{4}}\sqrt{|tn^{-2}|} \right. \\
& \left . - \sqrt{|tn^{-2}|} \alpha(tn^{-2}) - \beta(tn^{-2})) \right \} \\
& \frac{1}{n} \left [ \frac{c}{\sqrt{|tn^{-2}|}} + \sqrt{|tn^{-2}|} a(tn^{-2}) + b(tn^{-2}) \right ] - Q(t) \\
\end{split}
\end{displaymath}
Since $|tn^{-2}|<\epsilon$ and the quantity $x_n(t)$, defined by
\begin{displaymath}
x_n(t)=2e^{\textsf{sgn}(t) i \frac{\pi}{4}} \sqrt{|tn^{-2}|} + \sqrt{|tn^{-2}|} \alpha(tn^{-2}) + b(tn^{2}),
\end{displaymath}
goes to $0$ as $|tn^{-2}|$ goes to 0, developping the $\log$ yields
\begin{displaymath}
\begin{split}
D(t) & = \exp \left \{ -2e^{\textsf{sgn}(t) i \frac{\pi}{4}} \sqrt{|t|} \right \} \exp \left \{ -\sqrt{|t|} \alpha(tn^{-2}) - n \beta(tn^{-2}) \right \} \\
& e^{n x_n(t) \epsilon(x_n(t))} \frac{c}{\sqrt{|t|}} \left [ 1 + \frac{|t|a(tn^{-2})}{cn^{2}} + \frac{\sqrt{|t|}b(tn^{-2})}{cn} \right ] - Q(t) \\
\end{split}
\end{displaymath}
Now, we can factorize by $Q$
\begin{displaymath}
\begin{split}
D(t) = Q(t) & \left \{ \exp \left ( - \sqrt{|t|} \alpha(tn^{-2}) - n \beta(tn^{-2}) + nx_n(t)\epsilon(x_n(t)) \right ) - 1 \right . \\
& + \left . \left [ \frac{|t| a(tn^{-2})}{cn^{2}} + \sqrt{|t|} \frac{b(tn^{-2})}{cn} \right ] \right . \\
& \left . \exp \left ( - \sqrt{|t|} \alpha(tn^{-2}) - n \beta(tn^{-2}) + nx_n(t)\epsilon(x_n(t)) \right ) \right \},
\end{split}
\end{displaymath}
and take modulus,
\begin{displaymath}
\begin{split}
|D(t)| \leq |Q(t)| & \left | \exp \left \{ -\sqrt{|t|} \alpha(tn^{-2}) - n\beta(tn^{-2}) + n x_n(t) \epsilon(x_n(t)) \right \} - 1 \right | \\
& + |Q(t)| \left | \frac{|t|a(tn^{-2})}{cn^2} + \frac{\sqrt{|t|}b(tn^{-2})}{cn} \right | \\ 
& \left | \exp \left \{ -\sqrt{|t|} \alpha(tn^{-2}) - n\beta(tn^{-2}) + n x_n(t) \epsilon(x_n(t)) \right \} \right |. \\
\end{split}
\end{displaymath}
As a consequence, we have that
\begin{displaymath}
\begin{split}
\left | \frac{|t|a(tn^{-2})}{cn^{2}} + \frac{\sqrt{|t|}b(tn^{-2})}{c} \right | & = \sqrt{\frac{|t|}{n^2}} \left  | \sqrt{\frac{|t|}{n^2}} \frac{a(tn^{-2})}{c} + \sqrt{\frac{n^2}{|t|}} \frac{b(tn^{-2})}{c} \right | \\
& \leq N_{\epsilon} \sqrt{\frac{|t|}{n^2}}
\end{split}
\end{displaymath}
because the function $\rho(x) : x \mapsto \frac{\sqrt(x)}{c} a(x) + \frac{b(x)}{c\sqrt{x}}$ goes to $0$ as $x$ goes to $0$. The dependance to $\epsilon$ of $N_\epsilon$ is not so strong, we actually have uniformity --- due to the continuity of the function $\rho$ in the neighborhood of $0$ --- in the sense that there exists an $\epsilon_0 > 0$ such that for all $0 < \epsilon < \epsilon_0$ we have $N_\epsilon < N_{\epsilon_0}$. This uniformity will be interresting in the sequel. 

Using the following estimate,
\begin{displaymath}
|e^{a+ib}-1| \leq e^a |b| + |e^a-1|
\end{displaymath}
we have, for any $a \in \mathbb{R}$,
\begin{equation} \label{largeestimate}
|e^a-1| = |a| \left | \sum_{n \geq 1} \frac{a^{n-1}}{n!} \right | \leq |a| \sum_{n \geq 1} \frac{|a|^{n-1}}{n!}.
\end{equation}
Denoting by $\Upsilon(tn^{-2})$ the quantity
\begin{displaymath}
\Upsilon(tn^{-2})=\alpha(tn^{-2})+\frac{n}{\sqrt{|t|}}\beta(tn^{-2})+\frac{n}{\sqrt{|t|}} x_n(t) \epsilon(x_n(t)).
\end{displaymath}

The function $x \mapsto \alpha(x) + \frac{\beta(x)}{\sqrt{|x|}}$ is continuous at $x=0$ and
\begin{displaymath}
\begin{split}
|nx_n(t) \epsilon(x_n(t)) | \leq \sqrt{|t|} \left | 2e^{\textsf{sgn}(t) i \frac{\pi}{4}} + \alpha(tn^{-2}) + \frac{n}{\sqrt{t}} \beta(tn^{-2}) \right | |\epsilon(x_n(t))|
\end{split}
\end{displaymath}
but the function $\tilde{\rho} : x \mapsto 2e^{\textsf{sgn}(t) i \frac{\pi}{4}} + \alpha(x^2) + \frac{\beta(x)}{\sqrt{|x|}}$ is bounded so that
\begin{displaymath}
|nx_n(t) \epsilon(x_n(t))| \leq M \sqrt{|t|} K_\epsilon
\end{displaymath}
where $K_\epsilon$ comes from the fact that $\epsilon(x_n(t))$ goes to $0$ as $|tn^{-2}|$ goes to $0$, so that $|\epsilon(x_n(t))| \leq K_\epsilon$.
Summarising, $\Upsilon(tn^{-2})$ can be made arbitrarily small as $|tn^{-2}|$ goes to zero, namely $| \Upsilon(tn^{-2}) | \leq L_\epsilon$. Thus,
\begin{displaymath}
|e^{-\sqrt{|t|} \Upsilon(tn^{-2})}-1| \leq e^{-\sqrt{|t|} \mathsf{Re} \Upsilon(tn^{-2})} | \mathsf{Im} \sqrt{|t|} \Upsilon(tn^{-2}) | + |e^{-\sqrt{|t|} \mathsf{Re} \Upsilon(tn^{-2})} -1 |
\end{displaymath}
then, the first quantity is obviously majorized by
\begin{equation} \label{est1}
e^{-\sqrt{|t|} \mathsf{Re} \Upsilon(tn^{-2})} |\sqrt{|t|} \mathsf{Im} \Upsilon(tn^{-2}) | \leq e^{L_\epsilon \sqrt{|t|}} \sqrt{|t|} L_\epsilon 
\end{equation}
whereas for the second quantity, we use the estimate (\ref{largeestimate}) and we get
\begin{equation} \label{est2}
\begin{split}
|e^{-\sqrt{|t|} \mathsf{Re} \Upsilon(tn^{-2})}-1| \leq \sqrt{|t|} L_\epsilon \left | \sum_{k \geq 1} \frac{L_\epsilon^{k-1} |t|^{\frac{k-1}{2}}}{k!} \right |.
\end{split}
\end{equation}

Finally, it is obvious that $|e^z| \leq e^{|z|}$ for any complex number $z$, so that the following estimate holds
\begin{equation} \label{estimateD}
|D(t)| \leq |Q(t)| \left \{ e^{M \sqrt{|t|} K_\epsilon} N_{\epsilon_0} \sqrt{\frac{|t|}{n^2}} + e^{L_\epsilon \sqrt{|t|}} \sqrt{|t|} L_\epsilon + L_{\epsilon} \sqrt{|t|} \left | \sum_{k \geq 1} \frac{L_\epsilon^{k-1} |t|^{\frac{k-1}{2}}}{k!} \right | \right \}.
\end{equation}

Coming back to the proof of the proposition, we consider the first case, that is we suppose that $mn^{-2}$ converges to a real number, and we fix a $\delta > 0$ such that the decomposition in propositions \ref{dec} and \ref{decinitial} are satisfied. Then we can split
\begin{displaymath}
\begin{split}
n \int_{-\pi}^\pi e^{itm} \frac{\phi^n(t)}{1-\phi(t)} dt & = n \int_{-\delta}^\delta e^{itm} \frac{\phi^n(t)}{1-\phi(t)} dt + n \int_{|t|>\delta} e^{itm} \frac{\phi^n(t)}{1-\phi(t)} dt \\
& = I_1(m,n,\delta) + I_2(m,n,\delta). \\
\end{split}
\end{displaymath}
Let us consider first, the term $I_1(m,n,\delta)$, then setting $t=un^{-2}$ and decomposing as follows, we get
\begin{displaymath}
\begin{split}
n \int_{-\delta}^\delta e^{itm} \frac{\phi^{n}(t)}{1-\phi(t)} dt & = \int_{-n^2\delta}^{n^2\delta} e^{iumn^{-2}} \frac{\phi^n(un^{-2})}{n(1-\phi(un^{-2}))} du \\
& = \int_{-n^2\delta}^{n^2\delta} e^{iumn^{-2}} \left [ \frac{\phi^n(un^{-2})}{n(1-\phi(un^{-2}))} - \frac{\exp \{-2e^{\mathsf{sgn}(t)i \frac{\pi}{4}} \sqrt{|u|} \} }{\sqrt{|u|}} \right ] du \\
& + \int_{-\infty}^\infty e^{iumn^{-2}} \frac{\exp \{-2e^{\mathsf{sgn}(t)i \frac{\pi}{4}} \sqrt{|u|} \}}{\sqrt{|u|}} du \\
& - \int_{|u|>n^2\delta} e^{iumn^{-2}} \frac{\exp \{-2e^{\mathsf{sgn}(t) i \frac{\pi}{4}}\sqrt{|u|} \}}{\sqrt{|u|}} du. \\
& = I_3(m,n,\delta)+I_4(m,n)+I_5(m,n,\delta) \\
\end{split}
\end{displaymath}
It is easy to see that the term $I_5(m,n,\delta)$ converges to 0 as $n$ goes to infinity at the rate $\mathscr{O}(e^{-\sqrt{\frac{\pi}{2}}n})$ as the tail of the integral of an integrable function.

Applying the dominated convergence theorem to the term $I_4(m,n)$ implies that it converges to
\begin{displaymath}
  \int_{-\infty}^\infty e^{iu\lambda} \frac{\exp \{ -2 e^{ \mathsf{sgn}(t) i \frac{\pi}{4}} \sqrt{|u|} \} }{\sqrt{|u|}} du = s(\lambda)
\end{displaymath}
which is a non zero constant for all $\lambda$.

Finally, it remains to show that the term $I_3(m,n,\delta)$ goes to $0$. Using the estimate (\ref{estimateD}), we get
\begin{displaymath}
\begin{split}
\left | \int_{-n^2 \delta}^{n^2 \delta} e^{iumn^{-2}} \frac{\phi^n(un^{-2})}{n(1-\phi(un^{-2}))} du \right | & \leq \int_{-n^2 \delta}^{n^2 \delta} |Q(t)| \left \{  e^{M \sqrt{|t|} K_\epsilon} N_{\epsilon_0} \sqrt{\frac{|t|}{n^2}} \right . \\
& \left . e^{L_\epsilon \sqrt{|t|}} \sqrt{|t|} L_\epsilon + L_\epsilon \sqrt{|t|} \sum_{k \geq 1} \frac{L_\epsilon^{k-1} |t|^{\frac{k-1}{2}}}{k!} \right \} dt. \\
\end{split} 
\end{displaymath}
At this step, we have to choose $\epsilon > 0$ such that the decompositions \ref{dec} and \ref{decinitial} hold and such that $M K_\epsilon < \frac{\sqrt{2}}{4}$ and $L_\epsilon \leq \frac{\sqrt{2}}{2}$ so that the left handside integral  is majorized by
\begin{displaymath}
\begin{split}
\int_{-n^2\epsilon}^{n^2\epsilon} \frac{c}{n} N_{\epsilon_0} & e^{-\frac{\sqrt{2}}{4} \sqrt{|t|}} + c L_\epsilon e^{-\frac{\sqrt{2}}{2} \sqrt{|t|}} + cL_\epsilon e^{-\sqrt{2}{2}\sqrt{|t|}} \sum_{k \geq 1} \frac{\sqrt{2}^{k-1} |t|^{\frac{k-1}{2}}}{4^{k-1}k!} dt \\
& = I_6(n,\epsilon) + I_7(n,\epsilon) + I_8(n,\epsilon). \\
\end{split}
\end{displaymath}
Then, the quantity $I_6(n,\epsilon)$ goes to $0$ as $n$ goes to infinity, the quantity $I_7(n,\epsilon)$ can be made arbitrarily small, namely it behaves like a $\mathscr{O}(L_\epsilon)$, and setting $t=\frac{u^2}{2}$, $I_8(n,\epsilon)$ becomes
\begin{displaymath}
2L_\epsilon \int_{0}^{n \sqrt{2\epsilon}} e^{-u} \sum_{k \geq 1} \frac{\sqrt{2}^{k-1} |u|^k}{8^{k-1}k!} dt.
\end{displaymath}
Then, exchanging sum and integral, we get the majoration
\begin{displaymath}
2 L_\epsilon \sum_{k \geq 1} \frac{\sqrt{2}^{k-1}}{8^{k-1}k!} \int_0^\infty e^{-u} u^k du.
\end{displaymath}
But the latter integral is nothing but $(k+1)!$ thus the quantity $I_8(n,\epsilon)$ behaves like $\mathscr{O}(L_\epsilon)$ and as consequence it can be made arbitrarily small.

Finally, the term $I_2(m,n,\delta)$ goes to zero geometrically, and the proposition is proved.
\end{proof}

The following lemma is a refinement of a well known result on Fourier series.

\begin{lem} \label{fourier}
Let $(f_n)$ be a sequence of $2\pi$-periodic $\alpha$-Hölder real function with Hölder constante $Kn$ and $0<\alpha \leq 1$. Then for any $\epsilon > 0$, we have the estimate
\begin{displaymath}
\left | \int_{-\epsilon}^{\epsilon} f_n(t) e^{itm} dt \right | \leq \frac{Ln}{1+|m|^\alpha}
\end{displaymath}
for all $n,m \in \mathbb{Z}$.
\end{lem}

\begin{proof}
It is well known that 
\begin{displaymath}
\int_{-\pi}^\pi f_n(t) e^{itm} dt = \int_0^{2\pi} \frac{1}{m} \sum_{k=0}^{m-1} \left [ f_n \left ( \frac{t}{m}+\frac{2k\pi}{m} \right ) - f_n \left ( \frac{2k\pi}{m} \right ) \right ] e^{it} dt.
\end{displaymath}
Thus, we get that $\int_{-\epsilon}^\epsilon f_n(y) e^{itm} dt$ is given by
\begin{displaymath}
\int_0^{2\pi} \frac{1}{m} \sum_{k=0}^{m-1} \left [ f_n \left (\frac{t}{m}+\frac{2k\pi}{m} \right ) 1_{\Lambda_\epsilon} \left ( \frac{t}{m}+\frac{2k\pi}{m} \right )-f_n \left ( \frac{2k\pi}{m} \right )1_{\Lambda_\epsilon}\left ( \frac{2k\pi}{m} \right ) \right ]  e^{it} dt \\
\end{displaymath}
where $\Lambda_\epsilon=[\pi-\epsilon;\pi+\epsilon]$.

Then, the regularity of $f_n$ gives us that for any $x,y$
\begin{displaymath}
\begin{split}
|f_n(x)1_A(x)-f_n(y)1_A(y)| & \leq |f_n(x)-f_n(y)|+|f_n(y)||1_A(x)-1_A(y)| \\
& \leq Kn|x-y|^\alpha + M|1_A(x)-1_A(y)|. \\
\end{split}
\end{displaymath}
Consequently,
\begin{displaymath}
\begin{split}
& \left | \int_0^{2\pi} \frac{1}{m} \sum_{k=0}^{m-1} \left [ f_n \left (\frac{t}{m}+\frac{2k\pi}{m} \right )1_{\Lambda_\epsilon} \left (\frac{t}{m}+\frac{2k\pi}{m} \right )-f_n \left (\frac{2k\pi}{m} \right ) 1_{\Lambda_\epsilon} \left (\frac{2k\pi}{m} \right ) \right ]  e^{it} dt \right | \\
& \leq \int_0^{2\pi} \frac{1}{m} \sum_{k=0}^{m-1} Kn \left | \frac{t}{m} \right |^\alpha dt \\
& + \int_0^{2\pi} \frac{1}{m} \sum_{k=0}^{n-1} M \left | 1_{\Lambda_\epsilon}\left (\frac{t}{m}+\frac{2k\pi}{m} \right )-1_{\Lambda_\epsilon} \left (\frac{2k\pi}{m} \right ) \right | dt \\
& = J_1(m,\epsilon)+J_2(m,\epsilon). \\
\end{split}
\end{displaymath}
It is obvious that the quantity $J_1(m,\epsilon)$ is majorized by
\begin{displaymath}
\int_0^{2\pi} \frac{1}{m} \sum_{k=0}^{m-1} Kn \left | \frac{t}{n} \right |^\alpha dt \leq K' \frac{n}{|m|^\alpha}
\end{displaymath}
For the quantity $J_2(m,\epsilon)$, we only have to observe that the difference of indicator function is non zero for only two integers $k$, and, in that case, the difference is obviously bounded so that
\begin{displaymath}
\int_0^{2\pi} \frac{1}{m} \sum_{k=0}^{n-1} M \left | 1_{\Lambda_\epsilon}\left (\frac{t}{m}+\frac{2k\pi}{m} \right )-1_{\Lambda_\epsilon} \left (\frac{2k\pi}{m} \right ) \right | dt \leq \frac{2M}{m}.
\end{displaymath}
Therefore the lemma is proved.
\end{proof}

\begin{prp} \label{lemy1}
The sequence
\begin{displaymath}
\left (\sqrt{|y_1|} \int_{-\pi}^\pi e^{ity_1-itz} \frac{\phi^{y_2}(t)}{1-\phi(t)} dt \right )_{(y_1,y_2) \in \mathbb{Z}^2}
\end{displaymath}
converges to a non zero constant as $\frac{y_1}{y_2^2}$ goes to infinity.
\end{prp}

\begin{proof}
As in the previous proposition, set $n=y_2$ and $m=y_1-z$ for short. Thus, we want to estimate the integral
\begin{displaymath}
\int_{-\pi}^\pi e^{itm} \frac{g(r(t))^n}{1-\phi(t)} dt.
\end{displaymath}
Choose $\delta > 0$ so that the decompositions in propositions \ref{dec} and \ref{decinitial} are satisfied and split the integral,
\begin{displaymath}
\begin{split}
\int_{-\pi}^\pi e^{itm} \frac{g(r(t))^n}{1-\phi(t)} dt & = \int_{-\delta}^\delta e^{itm} \frac{g(r(t))^n}{1-\phi(t)} dt + \int_{|t|>\delta} e^{itm} \frac{g(r(t))^n}{1-\phi(t)} dt \\
& = I_1(m,n,\delta) + I_2(m,n,\delta). \\
\end{split}
\end{displaymath}
The function $t \mapsto g(r(t))^n [1-\phi(t)]^{-1}$ being continuously differentiable on the set $\{ |t| > \delta \}$, integrating by parts, we see that $I_2(m,n,\delta)$ goes to $0$ like $\mathscr{O}(\frac{n}{m})$ \emph{i.e.} like  $o\left (\frac{1}{\sqrt{|m|}}\right )$. 

Let us deal with the quantity $I_1(m,n,\delta)$, then we can write,
\begin{displaymath}
\begin{split}
\frac{g(r(t))^n}{1-\phi(t)} = & \frac{g(r(t))^n-(1-2e^{\textsf{sgn}(t) i \frac{\pi}{4}}\sqrt{|t|}-it)^n}{1-\phi(t)} \\
& + \frac{(1-2e^{\textsf{sgn}(t) i \frac{\pi}{4}}\sqrt{|t|}-it)^n-1}{1-\phi(t)} + \frac{1}{1-\phi(t)} \\
&  = R_1(n,t)+R_2(n,t)+R_3(n,t).
\end{split}
\end{displaymath}
We already know that the integral of function $R_3(n,t)$
\begin{displaymath}
\int_{-\delta}^\delta e^{itm} [1-\phi(t)]^{-1} dt
\end{displaymath}
is equivalent to the sequence $(c'|m|^{-1/2})_m$ as $|m|$ goes to infinity.
\fixme{$it/2$, la constante a-t-elle change ?}
Consider the function $R_1(n,t)$, then we can show it is Lipshitz with Lipshitz constant depending linearly on $n$. Let us denote by $q$ the function,
\begin{displaymath}
q(t)=1-2e^{\textsf{sgn}(t) i \frac{\pi}{4}}\sqrt{|t|}-it.
\end{displaymath}

Then, we split
\begin{displaymath}
\begin{split}
\frac{g(r(t))^n-q(t)^n}{1-\phi(t)} & = c \frac{g(r(t))^n-q(t)^n}{|t|^{1/2}} \\
& + \left [ g(r(t))^n-q(t)^n \right ] |t|^{1/2} a(t) \\
& + \left [ g(r(t))^n-q(t)^n  \right ] b(t). \\
\end{split}
\end{displaymath}
Actually, if the first quantity is continuously differentiable, the two other quantity are also continuously differentiable because they are obviously smoother. Let us compute the derivative of the first function.
\begin{displaymath}
\begin{split}
\frac{d}{dt} \frac{g(r(t))^n - q(t)^n}{|t|^{1/2}} & = \frac{d}{dt} \left ( \alpha(t) + \frac{\beta(t)}{\sqrt{|t|}} \right ) \sum_{k=0}^{n-1} g(t)^kq(t)^{n-k} \\
& = \left ( \alpha'(t) + \frac{\beta'(t)}{\sqrt{|t|}} - \frac{\beta(t)}{2|t|^{3/2}} \right ) \sum_{k=0}^{n-1} g(r(t))^k q(t)^{n-k} \\
& + \left ( \alpha(t) + \frac{\beta(t)}{\sqrt{|t|}} \right ) \sum_{k=0}^{n-1} \left [ k g'(t) g(r(t))^{k-1} q(t)^{n-k} \right . \\
& \left . + g(r(t))^k (n-k) q'(t) q(t)^{n-k-1} \right ]. \\
\end{split}
\end{displaymath}
The sums are estimated as follows,
\begin{displaymath}
\left | \sum_{k=0}^{n-1} g(r(t))^k q(t)^{n-k} \right | \leq n
\end{displaymath}
and
\begin{displaymath}
\left | \sum_{k=0}^{n-1}  k g'(r(t)) g(r(t))^{k-1} q(t)^{n-k} + g(r(t))^k (n-k) q'(t) q(t)^{n-k-1} \right | \leq \frac{Mn}{\sqrt{|t|}}
\end{displaymath}
where $M$ is a upper bound of $|\sqrt{|t|} g'(r(t))|$ and $|\sqrt{|t|} q'(t)|$ in a neighborhood of $0$.

Since, $\alpha(0)=\beta(0)=\beta'(0)=0$, the function $t \mapsto (\alpha'(t) + \beta'(t)|t|^{-1/2} - \beta(t) 2^{-1} |t|^{-3/2})$ is continous, therefore bounded. Moreover, the function $t \mapsto (\alpha(t) + \beta(t) |t|^{-1/2})$ is a $\mathscr{O}(\sqrt{|t|})$. Finally, we have the following estimate of the derivative,
\begin{displaymath}
\left | \frac{d}{dt} \frac{g(r(t))^n-q(t)^n}{|t|^{1/2}} \right | \leq Mn
\end{displaymath}
and this implies that the function $R_1(n,t)$ is Lipshitz with Lipshitz constant $Mn$.

By lemma \ref{fourier}, there exists a constant $K$ such that
\begin{displaymath}
\sqrt{|m|} \left | \int_{-\delta}^\delta \frac{g(r(t))^n - q(t)^n}{1-\phi(t)} e^{itm} dt \right | \leq \frac{Kn}{\sqrt{|m|}}
\end{displaymath}
so that the integral goes to $0$ when $\frac{\sqrt{|m|}}{n}$ goes to infinity.

It remains to estimate the integral of function $R_2(n,t)$, namely
\begin{displaymath}
\int_{-\delta}^\delta \frac{q(t)^n-1}{1-\phi(t)} e^{itm} dt
\end{displaymath}
which can be split as
\begin{displaymath}
\begin{split}
\int_{-\delta}^\delta c \frac{q(t)^n-1}{|t|^{1/2}} e^{itm} dt & + \int_{-\delta}^\delta (q(t)^n-1) |t|^{1/2} a(t)e^{itm} dt + \int_{-\delta}^\delta (q(t)^n-1) b(t) e^{itm}dt \\
& = I_3(m,n,\delta)+I_4(m,n,\delta)+I_5(m,n,\delta) \\
\end{split}
\end{displaymath}

Considering the integral $I_3(m,n,\delta)$, factorizing the quantity $q(t)^n-1$, and integrating by parts, we get
\begin{displaymath}
\begin{split}
\sqrt{|m|} \int_{-\delta}^\delta & c \frac{q(t)^n-1}{|t|^{1/2}} e^{itm} dt \\
& = -n\sqrt{|m|} \int_{-\delta}^\delta (2e^{\textsf{sgn}(t) i \frac{\pi}{4}} + i\sqrt{|t|}) \frac{1}{n} \sum_{k=0}^{n-1} (-1)^k (2e^{\textsf{sgn}(t) i\frac{\pi}{4}}+i \sqrt{|t|})^k e^{itm} dt \\
& = \frac{-n}{i\sqrt{|m|}} \left [ (2e^{\textsf{sgn}(t) i \frac{\pi}{4}} + i \sqrt{|t|}) \frac{1}{n} \sum_{k=0}^{n-1} (-1)^k (2e^{\textsf{sgn}(t) i \frac{\pi}{4}} + i \sqrt{|t|})^k e^{itm} \right ]_{-\epsilon}^{\epsilon} \\
& + \frac{n}{i\sqrt{|m|}} \int_{-\epsilon}^\epsilon \frac{d}{dt} \left [ (2e^{\textsf{sgn}(t) i \frac{\pi}{4}} + i \sqrt{|t|}) \frac{1}{n} \sum_{k=0}^{n-1} (-1)^k (2e^{\textsf{sgn}(t) i \frac{\pi}{4}} + i \sqrt{|t|})^k \right ] e^{itm} dt. \\
\end{split}
\end{displaymath}
The first quantity in the bracket is obviously bounded, so that the first term goes to $0$ as $\frac{\sqrt{|m|}}{n}$ goes to infinity. Consequently, it only remains to show that the derivative involved in the integral is integrable. Let us compute it,
\begin{displaymath}
\begin{split}
\frac{d}{dt} & \left [ (2e^{\textsf{sgn}(t) i \frac{\pi}{4}} + i \sqrt{|t|}) \frac{1}{n} \sum_{k=0}^{n-1} (-1)^k (2e^{\textsf{sgn}(t) i \frac{\pi}{4}} + i \sqrt{|t|})^k \right ] \\
& = \frac{i}{4\sqrt{|t|}} \frac{1}{n} \sum_{k=0}^{n-1} (-1)^k (2e^{\textsf{sgn}(t) i \frac{\pi}{4}} + i \sqrt{|t|})^k \\
& + (2e^{\textsf{sgn}(t) i \frac{\pi}{4}}+i \sqrt{|t|}) \frac{1}{n} \sum_{k=0}^{n-1} k (-1)^k ( \frac{e^{\textsf{sgn}(t) i \frac{\pi}{4}}}{\sqrt{|t|}} + i )(2e^{\textsf{sgn}(t) i \frac{\pi}{4}} + i \sqrt{|t|})^{k-1}. \\
\end{split}
\end{displaymath}
The Cesàro sum
\begin{displaymath}
\frac{1}{n} \sum_{k=0}^{n-1} k (-1)^k (2e^{\textsf{sgn}(t) i \frac{\pi}{4}} + i \sqrt{|t|})^{k-1}
\end{displaymath}
converges to $0$ as $n$ goes to infinity (hence is bounded). Thus, we get the estimate
\begin{displaymath}
\begin{split}
& \left | \frac{d}{dt} \left [ (2e^{\textsf{sgn}(t) i \frac{\pi}{4}} + i \sqrt{|t|}) \frac{1}{n} \sum_{k=0}^{n-1} (-1)^k (2e^{\textsf{sgn}(t) i \frac{\pi}{4}} + i \sqrt{|t|})^k \right ] \right | \\
& \leq \frac{1}{4\sqrt{|t|}} + K \left | 2e^{\textsf{sgn}(t) i \frac{\pi}{4}} + i \sqrt{|t|} \right | \left | \frac{e^{\textsf{sgn}(t) i \frac{\pi}{4}}}{\sqrt{|t|}} + i \right | \\
\end{split}
\end{displaymath}
and the latter is integrable.

Quantities $I_4(m,n,\delta)$ and $I_5(m,n,\delta)$ can be estimated in the same way and the proposition is proved.

\end{proof}

From proposition \ref{expnoyaufaible}, \ref{lemy2} and \ref{lemy1}, we get the following corollaries.

\begin{cor} \label{noyaufaible}
Let $z \in \mathbb{H}_0$, then we have
\begin{displaymath}
\lim_{|y| \to \infty} K(z,y) = 1.
\end{displaymath}
\end{cor}

Corollary \ref{noyaufaible} implies the following.
\begin{cor} \label{noyaufort}
Let $x \in \mathbb{H}$, then
\begin{displaymath}
\lim_{|y| \to \infty} \sum_{z \in \mathbb{H}_0} \nu_x(z) K(z,y) = 1
\end{displaymath}
\end{cor}

\begin{proof}
From Corollary \ref{noyaufaible} we have that, for any $z \in \mathbb{H}_0$,
\begin{displaymath}
\lim_{|y| \to \infty} K(z,y)=1.
\end{displaymath}

The sum $\sum_{(z,0) \in \mathbb{H}_0} \nu_x(z) K(z,y)$ is given by
\begin{displaymath}
\sum_{(z,0) \in \mathbb{H}_0} \nu_x(z) K(z,y) = \frac{\int_{-\pi}^\pi e^{ity_1} \phi^{y_2}(t) (1-\phi(t))^{-1} \sum \nu_x(z)e^{-itz} dt}{\int_{-\pi}^\pi e^{ity_1} \phi^{y_2} (1-\phi(t))^{-1} dt}.
\end{displaymath}
Noting that the probability $\nu_x(z)=\nu_{-x}(-z)$, the following equality holds 
\begin{displaymath}
\sum_{(z,0) \in \mathbb{H}_0} \nu_x(z)e^{-itz} = \sum_{(z,0) \in \mathbb{H}_0} \nu_{-x}(z)e^{itz}.
\end{displaymath} 
The latter is the characteristic function of $\nu_{-x}$ which is given (see the proof of theorem \ref{expnoyaufaible}) by
\begin{displaymath}
e^{-itx_1}\phi^{x_2}(t).
\end{displaymath}
Replacing in the integral, we obtain
\begin{displaymath}
\sum_{(z,0) \in \mathbb{H}_0} \nu_x(z)K(z,y) = \frac{\int_{-\pi}^\pi e^{it(y_1-x_1)} \phi^{|y_2|+|x_2|}(t) (1-\phi(t))^{-1} dt}{\int_{-\pi}^\pi e^{ity_1} \phi^{y_2} (1-\phi(t))^{-1} dt}
\end{displaymath}
and using the estimates of proposition \ref{lemy2} and \ref{lemy1}, one has the announced convergence.
\end{proof}

\subsection{Behavior before first return time}
\label{before}

Recall the equation (\ref{decomposition}) holding for $x,y \in \mathbb{H}$,
\begin{displaymath}
K(x,y)=\frac{\mathbf{E}^x(\eta_{0,\tau_1}(y))}{G(0,y)} + \sum_{z \in \mathbb{H}_0} \nu_x(z) K(z,y).
\end{displaymath}
It remains to show that the first term in this equation tends to zero. 

Assume that $x_2,y_2 \geq 0$ and $y_1 \geq x_1$ and let us fix our notation. We will define by $s^{\flat}_{y_i}$ for $i=1,2$ the following stopping time,
\begin{displaymath}
s^{\flat}_{y_i}=\inf \{ n \geq \flat : M_n^{(i)}=y_i, \forall k \leq n : M_n^{(i)} \neq 0 \}, \textrm{ with}, \flat \in \{ 0,1 \}.
\end{displaymath}
Then, we will denote by $g_u(y)$ the probability
\begin{displaymath}
g_u(y)=\mathbf{P}^{(y_1,u)}(s^0_{y_2} < \infty | M_{\tau_1}^{(1)} \geq y_1 ).
\end{displaymath}
Finally, the quantity $h_y$ will denote the probability
\begin{displaymath}
h_y=\mathbf{P}^{(y_1,y)}(s^1_y < \infty | M_{\tau_1}^{(1)} \geq y_1).
\end{displaymath}

\begin{prp} \label{decomp-esp}
The quantity $\mathbf{E}^x(\eta_{0,\tau_1}(y))$ is given by
\begin{displaymath}
\mathbf{E}^x(\eta_{0,\tau_1}(y)) = \frac{1}{(1-h_{y_2})^2} \sum_{u \geq 0} \mu_x(u) g_u(y_2)
\end{displaymath}
where $\mu_x$ is defined by
\begin{displaymath}
\mu_x(u)=\mathbf{P}^x(M_{s^0_{y_1}}=u, M_{\tau_1}^{(1)} \geq y_1).
\end{displaymath}
\end{prp}

\begin{proof}
It is a matter of fact that
\begin{displaymath}
\mathbf{E}^x(\eta_{0,\tau_1}(y))=\sum_{k \geq 0} k \mathbf{P}^x(\eta_{0,\tau_1}=k).
\end{displaymath}
On conditionning by the event $\{ s^0_{y_1} < \infty \}$, which is equal to the event $\{ M_{\tau_1}^{(1)} \geq y_1 \}$, we get
\begin{displaymath}
\mathbf{E}^x(\eta_{0,\tau_1}(y)) =\mathbf{P}^x(M_{\tau_1}^{(1)} \geq y_1) \sum_{k \geq 0} k \mathbf{P}^x(\eta_{0,\tau_1}(y)=k | M_{\tau_1}^{(1)} \geq y_1).
\end{displaymath}
By strong Markov property and observing that $s^0_{y_1}$ is finite on the event $\{ M_{\tau_1}^{(1)} \geq y_1 \}$, we get
\begin{displaymath}
\begin{split}
\mathbf{P}^x(\eta_{0,\tau_1}(y)=k | M_{\tau_1}^{(1)} \geq y_1) = \sum_{u \geq 0} \mathbf{P}^x( & M^{(2)}_{s^0_{y_1}}=u|M^{(1)}_{\tau_1} \geq y_1) \\
& \mathbf{P}^{(y_1,u)}(\eta_{0,\tau_1}(y)=k|M_{\tau_1}^{(1)} \geq y_1).
\end{split}
\end{displaymath}
Then, it is easy to see that
\begin{displaymath}
\mathbf{P}^{(y_1,u)}(\eta_{0,\tau_1}(y)=k | M_{\tau_1}^{(1)} \geq y_1)=g_u(y_2) h_{y_2}^{k-1}.
\end{displaymath}
Finally,  we get
\begin{displaymath}
\mathbf{E}^x(\eta_{0,\tau_1}(y)) = \sum_{k \geq 0} k \sum_{u \geq 0} \mu_x(u)g_u(y_2) h_{y_2}^{k-1}
\end{displaymath}
and grouping all terms, we obtain
\begin{displaymath}
\mathbf{E}^x(\eta_{0,\tau_1}(y)) = \frac{1}{1-h_{y_2}} \sum_{u \geq 0} \mu_x(u) g_u(y_2),
\end{displaymath}
proving thus the proposition.
\end{proof}

It is easy to get a upper bound for the probability $h_{y_2}$ because at the site $(y_1,y_2)$ it is possible to never come back with probability at least $1/3$ so that the quantity $(1-h_{y_2})^{-2}$ does not play any role in the asymptotics of the mean $\mathbf{E}^x(\eta_{0,\tau_1}(y))$.

\begin{prp} \label{exp-dec-g}
For any $u \geq 0$, the quantity $g_u(y)$ decreases exponentially fast to $0$ as $y$ goes to $\infty$.
\end{prp}

\begin{proof}
Recall that
\begin{displaymath}
g_u(y_2)=\mathbf{P}^{(y_1,u)}(s^0_{y_2} < \infty | M_{\tau_1}^{(1)} \geq y_1).
\end{displaymath}
Actually we can majorize $g_u$ by 
\begin{displaymath}
g_u(y_2) \leq \mathbf{P}^{(y_1,u)}(\exists n \geq 0 : M_n^{(2)}=y_2 | M_{\tau_1}^{(1)} \geq y_1) = p_u(y_2).
\end{displaymath}
Then, we can identify this probability with the probability to reach $y_2$ from $u$ in the model of a simple random walk on $\mathbb{Z}$ with a cemetery attached to each site, where the random walk can die with probability $1/3$. 

If we replace the cemetery by binary trees, then the probability $p_u$ satisfies
\begin{displaymath}
p_u(y_2) \leq F(u,y_2)
\end{displaymath}
where $F(u,y_2)$ is the probability to hit $y_2$ from $u$ in a homogenous tree of degree 3. By the lemma (1.24), found on p.9 of \cite{woessbook}, we get $F(u,y_2)=2^{-d(u,y_2)}$ where $d$ is the usual graph metric in the tree. Thus, $g_u(y_2)$ decreases exponentially fast to 0.

\end{proof}

\begin{prp} \label{ratesummux}
The quantity $\sum_{u \geq 0} \mu_x(u) g_u(y_2)$ behaves like $o(|y_2|^{-1})$ whenever $\frac{y_1}{y_2^2}$ converges to a finite limit and like $o(|y_1|^{-\frac{1}{2}})$ in other cases, namely when $\frac{y_1}{y_2^2}$ goes to $\pm \infty$.
\end{prp}

Before giving the proof of this fact, let us introduce some notation. We will denote by $(S_n)_{n \geq 0}$ the simple symmetric random walk on $\mathbb{Z}$. Recall that the characteristic function of $(S_n)$ starting from $z$ is given by
\begin{displaymath}
\mathbf{E}^z(e^{itS_n})=e^{itz}(e^{it(S_1-S_0)})^n= e^{itz}(\cos(t))^n
\end{displaymath}

On the set $\mathbb{N}$ we define the following Markov chain $(Z_n)_{n \geq 0}$ by its Markov operator $q : \mathbb{N} \times \mathbb{N} \mapsto [0,1]$ by
\begin{displaymath}
q(x,y)= \left \{ 
\begin{array}{ll} 
\frac{2}{3} & y=x \geq 1 \\
\frac{1}{3} & y=x-1, x \geq 1 \\
1 & x=y=0 \\
0 & \textrm{ otherwise } \\
\end{array} \right .
\end{displaymath}

On introducing the stopping time
\begin{displaymath}
T=\inf \{ n \geq 0 : Z_n=0 \},
\end{displaymath}
it is easy to compute its generating function.

\begin{lem} \label{genfunc}
The generation function of $T$ is given for any $h \geq 0$ by
\begin{displaymath}
\mathbf{E}^h(x^T)=\left (\frac{x}{3-2x} \right )^h.
\end{displaymath}
\end{lem}

\details{ \begin{proof}
Obviously, we have $\mathbf{E}^h(x^T)=\mathbb{E}^1(x^T)^h$, thus we only have to compute the generating function in the case $h=1$. Factorizing by the first step,
\begin{displaymath}
\mathbf{E}^1(x^T) = \sum_{k \geq 0} x^k \mathbf{P}^1(T=k) = \frac{2}{3}x\mathbf{E}^1(x^T) + \frac{1}{3}x
\end{displaymath}
then the lemma is proved.
\end{proof} }

We can now prove the proposition \ref{ratesummux}.

\begin{proof}
We can show that
\begin{displaymath}
\begin{split}
\mu_x(u)=\mathbf{P}^x(M^{(2)}_{s_{y_1}^0}=u, M^{(1)}_{\tau_1} \geq y_1) = \sum_{m \geq 0} \mathbf{P}^{x_2}( & S_m=u : S_k \neq 0, k \leq m) \\
& \mathbf{P}^1(T=m+(y_1-x_1)).
\end{split}
\end{displaymath}
Then, by the mirroring principle, we have that 
\begin{displaymath}
\mathbf{P}^{x_2}(S_m=u : S_k \neq 0, k \leq m)=\mathbf{P}^{x_2}(S_m=u)-\mathbf{P}^{-x_2}(S_m=u).
\end{displaymath}
Thus,
\begin{displaymath}
\begin{split}
\mu_x(u) & =\sum_{m \geq 0} \mathbf{P}^{x_2}(S_m=u)\mathbf{P}^1(T=m+(y_1-x_1)) \\
& - \sum_{m \geq 0} \mathbf{P}^{-x_2}(S_m=u)\mathbf{P}^1(T=m+(y_1-x_1)) \\
& = \Sigma_1(x,y,u) + \Sigma_2(x,y,u). \\
\end{split}
\end{displaymath}
Then, let us compute the sum $\Sigma_1(x,y,u)$,
\begin{equation} \label{mux1}
\begin{split}
\Sigma_1(x,y,u) & =\frac{1}{2\pi} \int_{-\pi}^\pi e^{itx_2} \sum_{m \geq 0} (\cos(t))^m \mathbf{P}^1(T=m+(y_1-x_1)) e^{-itu} dt \\
& = \frac{1}{2\pi} \int_{-\pi}^\pi F(\cos(t))^{y_1-x_1} e^{itx_2-itu} dt \\
\end{split}
\end{equation}
where $F(x)=\mathbf{E}^1(x^T)$ is the generating function of $T$. Whereas the sum $\Sigma_2(x,y_1,u)$ is given by
\begin{equation} \label{mux2}
\Sigma_2(x,y,u)=\frac{1}{2\pi} \int_{-\pi}^\pi F(\cos(t))^{y_1-x_1} e^{-itx_2-itu} du.
\end{equation}
As a consequence,
\begin{displaymath}
\mu_x(u)=\frac{1}{2\pi} \int_{-\pi}^\pi F(\cos(t))^{y_1-x_1} 2i \sin(tx_2) e^{-itu} dt.
\end{displaymath}
Now, from proposition \ref{exp-dec-g}, we get that
\begin{displaymath}
\sum_{u \geq 0} \mu_x(u) g_{y_2}(u) \leq \sum_{u \geq 0} \mu_x(u) 2^{-|y_2-u|}
\end{displaymath}
Split the sum
\begin{displaymath}
\begin{split}
\sum_{u \geq 0} \mu_x(u) 2^{-|y_2-u|} & = \sum_{u=0}^{y_2-1} \mu_x(u) 2^{-(y_2-u)} \\
& + \sum_{u=y_2}^\infty \mu_x(u) 2^{-(u-y_2)} \\
& = \Sigma_3(x,y)+\Sigma_4(x,y), \\
\end{split}
\end{displaymath}
and, injecting (\ref{mux1}) and (\ref{mux2}), sums $\Sigma_3(x,y)$ and $\Sigma_4(x,y)$ become
\begin{equation} \label{sum3}
\Sigma_3(x,y)=\frac{1}{2\pi} \int_{-\pi}^\pi F(\cos(t))^{y_1-x_1} 2i \sin(tx_2) \sum_{u=0}^{y_2-1} e^{-itu} 2^{u-y_2} dt.
\end{equation}
The geometric sum can be simplified by observing that
\begin{displaymath}
\sum_{u=0}^{y_2-1} e^{-itu}2^{u-y_2} = 2^{-y_2} \frac{(2e^{-it})^{y_2}-1}{2e^{-it}-1}
\end{displaymath}
hence, the sum (\ref{sum3}) becomes
\begin{equation} \label{finalsum3}
I_1(x,y)=\frac{1}{2\pi} \int_{-\pi}^\pi F(\cos(t))^{y_1-x_1} 2i \sin(tx_2) 2^{-y_2} \frac{(2e^{-it})^{y_2}-1}{2e^{-it}-1} dt.
\end{equation}

Similarly,
\begin{equation} \label{sum4}
\Sigma_4(x,y)=\frac{1}{2\pi} \int_{-2\pi}^\pi F(\cos(t))^{y_1-x_1} 2i\sin(tx_2) \sum_{u=y_2}^\infty e^{-itu} 2^{-(u-y_2)} dt,
\end{equation}
so that simplifying the geometric sum
\begin{displaymath}
\sum_{u=y_2}^\infty e^{-itu} 2^{-(u-y_2)} = e^{-ity_2} (1-\frac{e^{-it}}{2})^{-1}
\end{displaymath}
integral (\ref{sum4}) becomes
\begin{equation} \label{finalsum4}
I_2(x,y)= \frac{1}{2\pi} \int_{-\pi}^\pi F(\cos(t))^{y_1-x_1} 2i\sin(tx_2) e^{-ity_2} \frac{2}{2-e^{-it}} dt.
\end{equation}

At this step, it remains to study the rate of convergence of $I_1(x,y)$ and $I_2(x,y)$. We have to distinguish two cases depending on the way that $(y_1,y_2)$ goes to infinity :
\begin{itemize}
\item $y_1$ remains bounded ; 
\item $\lim \frac{y_2^2}{y_1} = \lambda$ for $\lambda \in \mathbb{R} \cup \{\pm \infty \}$ and $y_1$ is unbounded.
\end{itemize}

Let us handle the first case, and assume that $y_1$ is bounded. The function $F$ has a unique singularity for $x=\frac{3}{2}$ so that $F(\cos(\cdot))$ is infinitely continuously differentiable for $|t|\leq \pi$. As a consequence of lemma \ref{fourier}, the quantity $I_2(x,y)$ decreases like $\mathscr{O} \left (\frac{y_1^k}{y_2^k} \right )$ for arbitrary $k \geq 0$, \emph{i.e.} like $\mathscr{O} \left (\frac{1}{y_2^k} \right )$ because $y_1$ is supposed to be bounded. For the quantity $I_1(x,y)$, we have the following
\begin{displaymath}
\begin{split}
\int_{-\pi}^\pi F(\cos(t))^{y_1-x_1} 2i & \sin(tx_2) 2^{-y_2} \frac{(2e^{-it})^{y_2}-1}{2e^{-it}-1} dt \\ 
& = \int_{-\pi}^\pi F(\cos(t))^{y_1-x_1} \frac{2i\sin(tx_2)}{2e^{-it}-1} e^{-ity_2} dt \\
& - 2^{-y_2} \int_{-\pi}^\pi F(\cos(t))^{y_1-x_1} \frac{2i\sin(tx_2)}{2e^{-it}-1} dt
\end{split}
\end{displaymath}
Then, on one side, the first term goes to $0$ as $\mathscr{O}(\frac{1}{y_2^k})$ by lemma \ref{fourier} --- it is the same arguments as for the quantity $I_2(x,y)$ --- and on the other side, the second term goes obviously exponentially fast to $0$. Summarising, if $y_1$ remains bounded we have that
\begin{displaymath}
\mathbf{E}^x(\eta_{0,\tau_1}(y)) = \mathscr{O} \left (\frac{1}{y_2^k} \right)
\end{displaymath}
where $k$ is non negative and can be arbitrarily large.

Let us deal with the second case, and suppose that $y_1$ is unbounded. Rewriting the quantity $I_2(x,y)$ by setting $t=\frac{u}{\sqrt{|y_1|}}$, we get
\begin{equation} \label{split}
\begin{split}
& \frac{x_2}{y_1} 2i \int_{-\pi}^\pi F(\cos(t))^{y_1-x_1} \frac{y_1}{x_2} \sin(x_2 t) \frac{2}{2-e^{-it}} e^{-ity_2} dt \\
& = \frac{x_2}{y_1} 2i \int_{-\pi\sqrt{|y_1|}}^{\pi\sqrt{|y_1|}} F \left ( \cos\frac{t}{\sqrt{|y_1|}} \right )^{y_1-x_1} \frac{\sqrt{|y_1|}}{x_2} \sin \left (\frac{x_2 t}{\sqrt{|y_1|}} \right ) \frac{2e^{-i\frac{ty_2}{\sqrt{|y_1|}}}}{2-e^{-i\frac{t}{\sqrt{|y_1|}}}}  dt. \\
\end{split}
\end{equation}
Therefore,
\begin{displaymath}
\begin{split}
F \left ( \cos \frac{t}{\sqrt{|y_1}} \right )^{y_1-x_1} & = \exp \left \{-\frac{3}{2} \frac{y_1-x_1}{y_1} t^2 + \frac{y_1-x_1}{y_1}t^2 \epsilon \left ( \frac{t^2}{y_1} \right ) \right \} \\
& \longrightarrow e^{-\frac{3}{2} t^2} \textrm{ as } \frac{t^2}{|y_1|} \to 0, \\
\end{split}
\end{displaymath}
implying the following pointwise convergence,
\begin{displaymath}
F \left ( \cos\frac{t}{\sqrt{|y_1|}} \right )^{y_1-x_1} \frac{\sqrt{|y_1|}}{x_2} \sin \left (\frac{x_2 t}{\sqrt{|y_1|}} \right ) \frac{2}{2-e^{-i\frac{t}{\sqrt{|y_1|}}}} \longrightarrow e^{-\frac{3}{2}t^2}t
\end{displaymath}
as $\frac{t^2}{y_1} \to 0$. Let $\epsilon_0 > 0$ such that $\left | \frac{t^2}{y_1} \right | < \epsilon_0$, \emph{i.e.} $\left | \epsilon \left ( \frac{t^2}{y_1} \right ) \right | \leq \frac{3}{4}$. Then we get the domination
\begin{displaymath}
\begin{split}
\left | F \left ( \cos\frac{t}{\sqrt{|y_1|}} \right )^{y_1-x_1} \right . & \left . \frac{\sqrt{|y_1|}}{x_2} \sin \left (\frac{x_2 t}{\sqrt{|y_1|}} \right ) \frac{2}{2-e^{-i\frac{t}{\sqrt{|y_1|}}}} \right | \\
& \leq 2 M e^{-\frac{3}{2} t^2} |t| e^{\left | \frac{y_1-x_1}{y_1} \right | t^2 \left | \epsilon \left ( \frac{t^2}{y_1} \right ) \right |} \\
& \leq 2Me^{-\frac{3}{8} t^2} |t|.
\end{split}
\end{displaymath}
Consequently, we can split the integral (\ref{split}) as follows
\begin{displaymath} \label{qtyepsilon}
\begin{split}
\frac{x_2}{y_1} & 2i \int_{-\pi\sqrt{|y_1|}}^{\pi\sqrt{|y_1|}} F \left ( \cos\frac{t}{\sqrt{|y_1|}} \right )^{y_1-x_1} \frac{\sqrt{|y_1|}}{x_2} \sin \left (\frac{x_2 t}{\sqrt{|y_1|}} \right ) \frac{2e^{-i\frac{ty_2}{\sqrt{|y_1|}}}}{2-e^{-i\frac{t}{\sqrt{|y_1|}}}}  dt  \\
& = \frac{x_2}{y_1} 2i \int_{-\epsilon_0\sqrt{|y_1|}}^{\epsilon_0 \sqrt{|y_1|}} F \left ( \cos\frac{t}{\sqrt{|y_1|}} \right )^{y_1-x_1} \frac{\sqrt{|y_1|}}{x_2} \sin \left (\frac{x_2 t}{\sqrt{|y_1|}} \right ) \frac{2e^{-i\frac{ty_2}{\sqrt{|y_1|}}}}{2-e^{-i\frac{t}{\sqrt{|y_1|}}}}  dt \\
& + \frac{x_2}{y_1} 2i \int_{|t|>\sqrt{|y_1|} \epsilon_0} F \left ( \cos\frac{t}{\sqrt{|y_1|}} \right )^{y_1-x_1} \frac{\sqrt{|y_1|}}{x_2} \sin \left (\frac{x_2 t}{\sqrt{|y_1|}} \right ) \frac{2e^{-i\frac{ty_2}{\sqrt{|y_1|}}}}{2-e^{-i\frac{t}{\sqrt{|y_1|}}}}  dt \\
& = I_3(x,y) + I_4(x,y).
\end{split}
\end{displaymath}

The integral $I_3(x,y)$ converges by Lebesgue convergence to the integral
\begin{displaymath}
\int_{-\infty}^\infty e^{-\frac{3}{2}t^2} t e^{-it \lambda} dt
\end{displaymath}
with $\lambda = \lim \frac{y_2}{\sqrt{|y_1|}}$. And this integral can be easily computed,
\begin{displaymath}
\int_{-\infty}^\infty e^{-\frac{3}{2}t^2} t e^{-it\lambda} dt = \frac{i\lambda}{3} \int_{-\infty}^\infty e^{-\frac{3}{2}t^2} e^{-it\lambda} dt = \frac{i\lambda}{3} \sqrt{\frac{2\pi}{3}} e^{-\frac{\lambda^2}{6}}.
\end{displaymath}
Then substituting $\lambda$ by the ratio $\frac{y_2^2}{y_1}$ the quantity $(12)$ becomes
\begin{displaymath}
- \frac{2}{3} \frac{x_2}{y_1} \sqrt{\frac{2\pi}{3}} \frac{y_2}{\sqrt{|y_1|}} e^{-\frac{1}{6} \frac{y_2^2}{y_1}}.
\end{displaymath}
We conclude that,
\begin{itemize}
\item if $\frac{y_2^2}{y_1}$ goes to $0$, then $I_3(x,y)$ behaves like $o \left (\frac{1}{\sqrt{|y_1|}} \right )$ ;
\item if $\frac{y_2^2}{y_1}$ goes to $\pm \infty$, $I_3(x,y)$ behaves like $o \left (\frac{1}{|y_2|} \right )$ ;
\item finally, if $\frac{y_2^2}{y_1}$ converges to $\lambda$ non zero real, then $I_3(x,y)$ behaves again like $o(\frac{1}{|y_2|})$.
\end{itemize}

Integrating by parts gives us the following estimate of $I_4(x,y)$,
\begin{displaymath}
\begin{split}
\left | \frac{x_2}{y_1} 2i \right . & \left . \int_{|t|>\sqrt{|y_1|} \epsilon_0} F \left ( \cos\frac{t}{\sqrt{|y_1|}} \right )^{y_1-x_1} \frac{\sqrt{|y_1|}}{x_2} \sin \left (\frac{x_2 t}{\sqrt{|y_1|}} \right ) \frac{2e^{-i\frac{ty_2}{\sqrt{|y_1|}}}}{2-e^{-i\frac{t}{\sqrt{|y_1|}}}}  dt \right | \\
& \leq \frac{My_1}{y_2} L^{y_1-x_1}
\end{split}
\end{displaymath}
because,
\begin{displaymath}
\sup_{|t|>\epsilon_0} \left | \frac{d}{dt} F(\cos(t)) \right | < 1.
\end{displaymath} 
As a consequence, the quantity $I_4(x,y)$ behaves like
\begin{itemize}
\item $o \left (\frac{1}{|y_2|} \right )$ if $\frac{y_1}{y_2^2}$ converges to a finite limit with $y_1$ unbounded.
\item $o \left ( \frac{1}{\sqrt{|y_1|}} \right )$ if $\frac{y_1}{y_2^2}$ goes to $\textsf{sgn}(t) \infty$.
\end{itemize}

Turning to the quantity $I_1(x,y)$, we note that
\begin{displaymath}
\begin{split}
I_1(x,y)=\int_{-\pi}^\pi & F(\cos(t))^{y_1-x_1} 2i \sin(tx_2) 2^{-y_2} \frac{(2e^{-it})^{y_2}-1}{2e^{-it}-1} dt \\ 
& = \int_{-\pi}^\pi F(\cos(t))^{y_1-x_1} 2i \sin(tx_2)\frac{e^{-ity_2}-1}{2e^{-it}-1} dt \\
& - 2^{-y_2} \int_{-\pi}^\pi F(\cos(t))^{y_1-x_1} 2i \sin(tx_2)\frac{1}{2e^{-it}-1} dt \\
= I_5(x,y)+I_6(x,y).
\end{split}
\end{displaymath}
The quantity $I_5(x,y)$ can be estimated the same way the quantity $I_2(x,y)$ is whereas the quantity $I_6(x,y)$ behaves like $o \left (\frac{1}{|y_2|} \right)$ in the case where $\frac{y_1}{y_2^2}$ converges to finite limit. It remains to show that $I_6(x,y)$ behaves like $o \left (\frac{1}{\sqrt{|y_1|}} \right )$ in the case where $\frac{y_1}{y_2^2}$ goes to infinity. We can estimate the integral $I_6(x,y)$ the same way it has been done for the quantity $I_2(x,y)$ in the case of $y_1$ unbounded and $y_2$ fixed.
\end{proof}

Obviously, by symmetry, all these estimations can be made in the case $x_2,y_2 \leq 0$ and $y_1 \leq x_1$. And as soon as, $x_2y_2 < 0$ then the mean $\mathbf{E}^x(\eta_{0,\tau_1}(y))$ is zero, therefore we get the following.

\begin{cor} \label{ccl-firstt}
The quantity
\begin{displaymath}
\frac{\mathbf{E}^x(\eta_{0,\tau_1}(y))}{G(0,y)}
\end{displaymath}
in equation (\ref{decomposition}) goes to $0$ when $|y|$ goes to infinity.
\end{cor}

\begin{proof}
By propositions \ref{lemy2}, \ref{lemy1} and \ref{ratesummux}, we have that
\begin{itemize}
\item $G(o,y)$ is equivalent to $ \left (\frac{c}{\sqrt{|y_1|}} \right )$ if $\frac{y_1}{y_2^2}$ goes to infinity ;
\item $G(0,y)$ is equivalent to $\left ( \frac{c'}{|y_2|} \right )$ if $\frac{y_1}{y_2^2}$ converges to a finite limit.
\end{itemize}
In the first case, the quantity
\begin{displaymath}
\mathbf{E}^x(\eta_{0,\tau_1}(y))=o \left (\frac{1}{\sqrt{|y_1|}} \right )
\end{displaymath}
and in the second case,
\begin{displaymath}
\mathbf{E}^x(\eta_{0,\tau_1}(y))=o \left (\frac{1}{|y_2|} \right ).
\end{displaymath}
Then, obviously, the ratio involved in the corollary converges to $0$ in any direction as $|y|$ goes to infinity.
\end{proof}

\begin{proof}[Proof of theorem \ref{original}]
Since for all $x \in \mathbb{H}$, $K(x,y_k)$ has no other limit point than $1$ for all unbounded sequence $(y_k)$ then, the Martin boundary is trivial.
\end{proof}

\subsection{Proofs of analytic decompositions} \label{technical}

\begin{lem}
The function $\phi$ is given by
\begin{displaymath}
\begin{split}
\phi & (t) = \frac{1}{p^2}(1-2q\cos(t)+q^2 \cos(2t)) \\
& - \Bigg [ \sqrt{1-2q\cos(t)+q^2} \left ( (p^{-1}-1)^2-\frac{2q}{p}(p^{-1}-1)\cos(t)+\frac{q^2}{p^2} \right )^{\frac{1}{4}} \\
& \left ( (p^{-1}+1)^2-\frac{2q}{p}(p^{-1}+1)\cos(t)+\frac{q^2}{p^2} \right )^{\frac{1}{4}} \cos \left [ \arctan \left ( \frac{-q\sin(t)}{1-q\cos(t)} \right ) \right . \\
& \left . +\frac{1}{2}\arctan \left ( \frac{-\sin(t)}{1-\cos(t)} \right ) + \frac{1}{2} \arctan \left ( \frac{-q\sin(t)}{1+p-q\cos(t)} \right ) \right ] \Bigg ].
\end{split}
\end{displaymath}
Furthermore, in the case of the simple random walk we have $p=1/3=1-q$, so that
\begin{equation} \label{phi-expression}
\begin{split}
\phi(t) & = (9-12\cos(t)+4\cos(2t)) \\
& - \Bigg [ \frac{\sqrt{13-12\cos(t)}}{3} 8^{1/4} \left ( (1-\cos(t) \right )^{\frac{1}{4}} 4^{1/4} \left ( 5-4\cos(t) \right )^{\frac{1}{4}} \\
& \cos \left [ \arctan \left ( \frac{-2\sin(t)}{3-2\cos(t)} \right )+\frac{1}{2}\arctan \left ( \frac{-\sin(t)}{1-\cos(t)} \right ) \right . \\
& \left . +\frac{1}{2} \arctan \left ( \frac{-\sin(t)}{2-\cos(t)} \right ) \right ] \Bigg ].
\end{split}
\end{equation}
\end{lem}

\begin{proof}
Denote by $z$ the complex number $z=1-qe^{it}$, then we get
\begin{displaymath}
\phi(t)=\mathsf{Re} \frac{z^2}{p^2}-\frac{z}{p} \sqrt{\frac{z^2}{p^2}-1}
\end{displaymath}
A simple computation gives us that $\mathsf{Re} \frac{z^2}{p^2} = \frac{1}{p^2}(1-2q \cos(t) + q^2 \cos(2t))$. It remains to make explicit the term with the square root. Start by expanding in polar form,
\begin{displaymath}
\frac{z}{p} \sqrt{\frac{z^2}{p^2}-1} = \frac{z}{p} \sqrt{\frac{z}{p}-1} \sqrt{\frac{z}{p}+1},
\end{displaymath}
then, we have for the modulus of $z$,
\begin{displaymath}
|z|^2 = 1-2q\cos(t)+q^2,
\end{displaymath}
and for its argument
\begin{displaymath}
\mathsf{Arg} \left ( \frac{z}{p} \right )=\arctan \left ( \frac{-q\sin(t)}{1-q\cos(t)} \right ).
\end{displaymath}

For the modulus and argument of $\frac{z}{p}-1$
\begin{displaymath}
\left | \frac{z}{p}-1 \right |^2=\left | \frac{1}{p}-1-\frac{q}{p} e^{it} \right |^2=(p^{-1}-1)^2-\frac{2q}{p}(p^{-1}-1)\cos(t)+\frac{q^2}{p^2},
\end{displaymath}
and
\begin{displaymath}
\mathsf{Arg}\left ( \frac{z}{p}-1 \right )=\arctan \left ( \frac{-\sin(t)}{1-\cos(t)} \right ).
\end{displaymath}

Finally, we have for $\frac{z}{p}+1$
\begin{displaymath}
\left | \frac{z}{p}+1 \right |^2= \left |\frac{1}{p}+1-\frac{q}{p} e^{it} \right |^2=(p^{-1}+1)^2-\frac{2q}{p}(p^{-1}+1)\cos(t)+\frac{q^2}{p^2},
\end{displaymath}
and
\begin{displaymath}
\mathsf{Arg} \left ( \frac{z}{p}+1 \right )=\arctan \left ( \frac{-q\sin(t)}{1+p-q\cos(t)} \right ).
\end{displaymath}
\end{proof}

\begin{proof}[Proof of proposition \ref{decinitial}.]
It is easy to show that
\begin{displaymath}
\frac{-\sin(t)}{1-\cos(t)}=-\frac{2}{t}(1+A_0(t))
\end{displaymath}
and that the power series of $\arctan$ in the neighborhood of $-\infty$ and $+\infty$ andis given by
\begin{displaymath}
\arctan(v)=\pm \frac{\pi}{2} - \sum_{n \geq 0} (-1)^n \frac{1}{(2n+1)v^{2n+1}}
\end{displaymath}
and the $\pm$ depends on the fact that $v$ is in the neighborhood of $\pm \infty$. Consequently, it gives
\begin{displaymath}
\arctan \left ( \frac{-\sin(t)}{1-\cos(t)} \right ) = \textsf{sgn}(t) \frac{\pi}{2} - \frac{t}{2}(1-A_1(t))
\end{displaymath}
with $A_1$ analytic such that $A_1(0)=0$.

The functions $t \mapsto \arctan \left ( \frac{-2\sin(t)}{3-2\cos(t)} \right )$ and $t \mapsto \arctan \left ( \frac{-\sin(t)}{2-\cos(t)} \right )$ are analytic in a neighborhood of $0$ and vanishes for $t=0$. Thus, the expansion in a power series of the cosine in equation \ref{phi-expression} is given by $\frac{\sqrt{2}}{2}(1+A_2(t))$ where $A_2$ is analytic and $A_2(0)=0$.

The only remaining problematic term is $(1-\cos(t))^{1/4}$ which can rewritten as $\sqrt{|t|} A_3(t)$ with $A_3$ a power series around $0$.

Summarizing, there exists two analytic functions $A_4(t)$ and $A_5(t)$ such that
\begin{displaymath}
\phi(t)=1-\sqrt{|t|}A_4(t)-A_5(t)
\end{displaymath}
and the proposition \ref{decinitial} easily follows.
\end{proof}


\begin{proof}[Proof of proposition \ref{dec}.]
We already know that $g(r(t))$ is given by
\begin{displaymath}
g(r(t))=\frac{1-\sqrt{1-r(t)^2}}{r(t)}=\frac{1}{r(t)}-\sqrt{\frac{1}{r(t)^2}-1}.
\end{displaymath}
The first term is very easy to decompose
\begin{displaymath}
\frac{1}{r(t)}=3-2e^{it}=1+2(1-e^{it})=1-\beta(t)
\end{displaymath}
where $\beta$ is given by $\beta(t)=2\sum_{n \geq 1} \frac{(it)^n}{n!}$.

The second term with the square root requires a finer analysis. First we have to express the argument of the square in polar form.
\begin{displaymath}
\frac{1}{r(t)^2}-1 =(3-2e^{it})^2-1 = 4(2-e^{it})(1-e^{it})
\end{displaymath}
Then, we compute the square of the modulus,
\begin{displaymath}
\left | \frac{1}{r(t)^2} - 1 \right |^2 = 32(5-4\cos(t))(1-\cos(t))
\end{displaymath}
Thus, the square root of the modulus is given by
\begin{displaymath}
\sqrt{\left | \frac{1}{r(t)^2}-1 \right |}=2\sqrt{|t|}(1+A_0(t))
\end{displaymath}
where $A_0(t)$ is an analytic funtion satisfying $A_0(0)=0$.

Let us now decompose the argument of the complex function $r(t)^{-2}-1$, 
\begin{displaymath}
\arg \left (\frac{1}{r(t)^2}-1 \right ) = \arctan \frac{-\sin(t)}{2-\cos(t)}
+\arctan \frac{-\sin(t)}{1-\cos(t)}.
\end{displaymath}

The first term $\arctan \frac{-\sin(t)}{2-\cos(t)}$ is analytic as the composition of two analytic functions.

For the second term, we compute as in the proof of the proposition \ref{decinitial}
\begin{displaymath}
\arctan \left ( \frac{-\sin(t)}{1-\cos(t)} \right ) = \textsf{sgn}(t) \frac{\pi}{2} - \frac{t}{2}(1-A_1(t))
\end{displaymath}
with $A_1$ analytic such that $A_1(0)=0$.

Finally we get the following decomposition,
\begin{displaymath}
\sqrt{\frac{1}{r(t)^2}-1} = \sqrt{|t|} (1+A_0(t))e^{i\textsf{sgn}(t)\frac{\pi}{4}}A_7(t)
\end{displaymath}
with $A_0(0)=0$ and $A_2(0)=1$ and letting $\alpha(t)=A_0(t)A_2(t)e^{\textsf{sgn}(t) i \frac{\pi}{4}}$, the proposition is proved.
\end{proof}
\subsection{Poisson boundary} \label{poisson}

In this section, we give an elementary proof of the triviality of the Poisson of the simple random walk on $\mathbb{H}$ even though we already know that since the Martin boundary is trivial. However, the ideas in this elementary proof can be exploited to show the triviality of the Poisson boundary for more general random walks and orientations for which the description of the Martin boundary would be tedious.
\newline

\textbf{The case of the simple random walk}
\newline

The following proposition is proved by adapting the proof of the triviality of the Poisson boundary of random walks on Abelian groups due to Choquet and Deny (see \cite{choquet-convolution}) or more specifically we will adapt the proof of theorem T1, chapter VI, in \cite{Spi}.

\begin{prp} \label{poisson-srw}
The Poisson boundary of the simple random walk on $\mathbb{H}$ is trivial, \emph{i.e} all bounded harmonic function are constant.
\end{prp}

\begin{proof}[Elementary proof]
Let $h$ be a bounded harmonic function and $a=(\alpha,0)$ a vector of $\mathbb{H}$. We set $g(x)=h(x)-h(x-a)$, then $g$ is obviously harmonic
\begin{displaymath}
Pg(x)=h(x)-\sum_{y \in \mathbb{H}} p(x,y)h(y-a).
\end{displaymath}
Thus, setting $z=x-a$, substituting in the sum, and noting that $p(x,z+a)=p(x-a,z)$ because $P$ is invariant by horizontally translation, we get
\begin{displaymath}
Pg(x)=h(x)-\sum_{z \in \mathbb{H}} p(x-a,z)h(z)=h(x)-h(x-a)=g(x).
\end{displaymath}
Now let $\sup_{x \in \mathbb{H}} g(x)=M < \infty$, choose a sequence $x_n$ of point in $\mathbb{H}$ such that
\begin{displaymath}
\lim_{n \to \infty} g(x_n)=M,
\end{displaymath}
and let
\begin{displaymath}
g_n(x)=g(x+x_n).
\end{displaymath}
Since $g$ is bounded, one can select a subsequence $x_n^{(1)}$ from the sequence $x_n$ such that, for a certain $x=x_1$
\begin{displaymath}
\lim_{n \to \infty} g(x_1+x_n^{(1)}) \textrm{ exists. }
\end{displaymath}
However, we can do better. We can take a subsequence $x_n^{(2)}$ of the sequence $x_n^{(1)}$ such that $g(x+x_n^{(2)})$ has a limit at $x=x_1$ and also at $x=x_2$. This process can be continued. By the Cantor's diagonalisation principle, $\mathbb{H}$ being countable, there exists a subsequence $n_l$ of positive integers and a real function $g^*$ on $\mathbb{H}$ such that
\begin{displaymath}
\lim_{l \to \infty} g_{n_l}=g^*(x)
\end{displaymath}
for every $x \in \mathbb{Z}$. Moreover, it is obvious that
\begin{displaymath}
g^*(0)=M, \textrm{ and  } g^*(x) \leq M \textrm{ for all } x \in \mathbb{H}.
\end{displaymath}
Furthermore, the function $g^*$ is harmonic by dominated convergence.

Recall that the simple random walk on $\mathbb{H}$ is irreducible because the graph is connected. Thus, applying the maximum principle to the harmonic function implies that $g^*\equiv g^*(0)=M$. 

Let $r$ be any positive integer and $\epsilon > 0$, we can find an integer $n$ large enough such that
\begin{displaymath}
g_n(a) > M-\epsilon \textrm{ ; } g_n(2a) > M-\epsilon \textrm{ ; } \cdots \textrm{ ; } g_n(ra) > M-\epsilon.
\end{displaymath}
Going back to the definition of $g_n$ and adding those $r$ inequalities, we obtain
\begin{displaymath}
h(ra+x_n)-h(x_n) > r(M-\epsilon)
\end{displaymath}
for large $n$ enough. We can show that $M$ can not be positive. Indeed, if it was, the integer $r$ could have been chosen so large that $r(M-\epsilon)$ exceeds the least upper bound of $h$. Therefore, it follows $g(x) \leq M \leq 0$ and $h(x) \leq h(x-a)$. Obviously, we can do the same reasoning for $-h$ and we would have $h(x) \geq h(x-a)$.

Setting $\tilde{h}(y)=h(x_0,y)$ for some $x_0$, we show that the bounded harmonic function $\tilde{h}$ is constant by maximum principle.
\end{proof}
\quad

\textbf{The case of random walk on $\mathbb{H}$ with a drift}
\newline

Looking at the proof of the proposition \ref{poisson-srw}, we observe that the crucial property is the translation invariance of the operator which allows to consider the simpler problem of the determination of the bounded harmonic functions associated with a specific random walk on $\mathbb{Z}$.

Let $(p_y)_{y \in \mathbb{Z}}$ be a sequence of real number in $[0,1)$ and let $(q_y)_{y \in \mathbb{Z}}$ be a sequence of positive real numbers with $q_y<1-p_y$ for all $y \in \mathbb{Z}$. We suppose that, at the site $(x,y) \in \mathbb{H}$, the random walk can move horizontally with probability $p_{y}$, move up with probability $q_y$ and move down with probability $1-p_y-q_y$ (figure \ref{fig-irreducible}). Bearing in mind what we have noticed, the following theorem does not require a proof.

\begin{figure}[!h] 
\begin{center}
\includegraphics[width=0.5\textwidth]{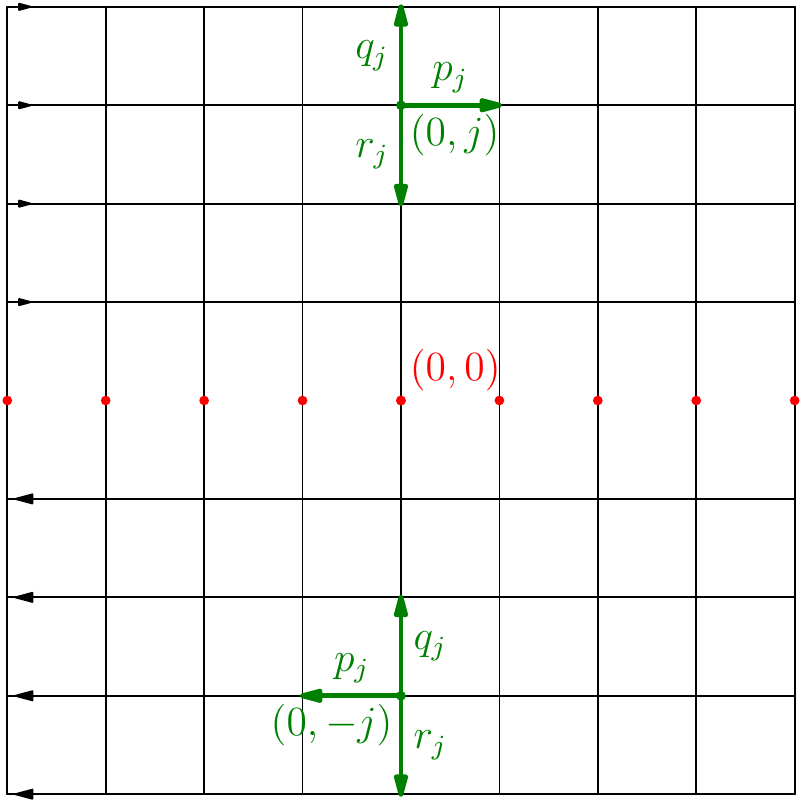}
\end{center}
\caption{The half plane one-way lattice $\mathbb{H}$ with a non constant drift.} \label{fig-irreducible}
\end{figure} 

\begin{thm}
The Poisson boundary of the random walk on $\mathbb{H}$ with transition probabilities defined as above is isomorphic to the Poisson boundary of the random walk whose transition operator is defined for $x,y \in \mathbb{Z}$ by
\begin{displaymath}
p(x,y)=\left \{ \begin{array}{ll}
p_x & \textrm{ if } $y=x$, \\
q_x & \textrm{ if } $y=x+1$, \\
1-p_x-q_x & \textrm{ if } $y=x-1$, \\
0 & \textrm{ otherwise.} \\
\end{array} \right .
\end{displaymath}
\end{thm}

In our context, the orientation $\epsilon$ has been fixed once for all. However, it can be chosen randomly. If $\epsilon=(\epsilon_y)_{y \in \mathbb{Z}}$ is a sequence of independent random variables it is shown in \cite{Pet1} that the corresponding simple random walk on $(\mathbf{G},\epsilon)$ is transition for almost all $\epsilon$. This result has been generalized in \cite{GP-functional} for a random sequence $\epsilon$ for which $\epsilon_y$ is equal to 1 with probability $f_y$ and -1 with probability $1-f_y$ where $(f_y)_{y \in \mathbb{Z}}$ is a sequence of stationary random variables satisfying $\mathbf{E}(f_0(1-f_0))^{-1/2} < \infty$. Finally, the case of a stationary sequence $\epsilon$ with decorrelation conditions is considered in \cite{pene2009transient} and, also, the corresponding simple random walk is shown to be transient. In those situations, the Poisson boundary remains obviously trivial (for all orientations) since, for all $y \in \mathbb{Z}$, $q_y \equiv p_y = \frac{1}{3}$ and the corresponding Markov operator on $\mathbb{Z}$ is invariant par the natural $\mathbb{Z}$-action.

\bibliographystyle{alpha}
\bibliography{biblio}

\begin{thebibliography}{GPLN08}

\bibitem[CD60]{choquet-convolution}
Gustave Choquet and Jacques Deny.
\newblock Sur l'\'equation de convolution {$\mu =\mu \ast \sigma $}.
\newblock {\em C. R. Acad. Sci. Paris}, 250:799--801, 1960.

\bibitem[CP03]{Pet1}
M.~Campanino and D.~Petritis.
\newblock Random walks on randomly oriented lattices.
\newblock {\em Markov Process. Related Fields}, 9(3):391--412, 2003.

\bibitem[GPLN08]{GP-functional}
Nadine Guillotin-Plantard and Arnaud Le~Ny.
\newblock A functional limit theorem for a 2{D}-random walk with dependent
  marginals.
\newblock {\em Electron. Commun. Probab.}, 13:337--351, 2008.

\bibitem[P{\`e}n09]{pene2009transient}
F.~P{\`e}ne.
\newblock Transient random walk in with stationary orientations.
\newblock {\em ESAIM: Probability and Statistics}, 13:417--436, 2009.

\bibitem[Spi76]{Spi}
Frank Spitzer.
\newblock {\em Principles of random walks}.
\newblock Springer-Verlag, New York, second edition, 1976.
\newblock Graduate Texts in Mathematics, Vol. 34.

\bibitem[Woe09]{woessbook}
Wolfgang Woess.
\newblock {\em Denumerable {M}arkov chains}.
\newblock EMS Textbooks in Mathematics. European Mathematical Society (EMS),
  Z\"urich, 2009.
\newblock Generating functions, boundary theory, random walks on trees.

\end{thebibliography}

\end{document}